\documentclass[a4paper,12pt]{amsart}

\usepackage{fullpage}
\usepackage{amsmath,amssymb}
\usepackage[T1]{fontenc}
\usepackage{graphicx,psfrag,ifthen,moreverb}
\usepackage{color}
\usepackage{subfigure}

\definecolor{darkgreen}{rgb}{0,0.45,0}

\definecolor{darkblue}{rgb}{0,0,.8}
\definecolor{darkred}{rgb}{0.8,0,0}
\usepackage{hyperref}
\hypersetup{
  pdftex=true,
  pdfborder={0 0 0},
}

\newcommand{\A}{\mathbb{A}}

\newcommand{\N}{\mathbb{N}}
\newcommand{\R}{\mathbb{R}}
\newcommand{\OO}{\mathcal{O}}
\newcommand{\RR}{\mathcal{R}}
\newcommand{\TT}{\mathcal{T}}

\newcommand{\XX}{\mathcal{X}}
\newcommand{\II}{\mathcal{I}}

\newcommand{\enorm}[2][]{#1|\!#1|\!#1|\,#2\,#1|\!#1|\!#1|}
\newcommand{\norm}[3][]{#1\|#2#1\|_{#3}}
\newcommand{\diam}{{\rm diam}}
\newcommand{\set}[3][\big]{#1\{#2\,:\,#3#1\}}

\def\reff#1#2{\!\stackrel{\eqref{#1}}{#2}\!}

\def\UU{\mathcal U}
\def\MM{\mathcal M}
\def\refine{\operatorname{refine}}
\def\Cmark{C_{\rm mark}}
\def\Cstab{C_{\rm stb}}

\def\Crel{C_{\rm rel}}
\def\qred{q_{\rm red}}
\def\T{\mathbb{T}}

\def\qlin{q_{\rm lin}}

\def\Cdrel{C_{\rm drel}}
\def\qref{q_{\rm ref}}%

\def\nf{{\mathbf n}}

\def\Cmon{C_{\rm mon}}
\def\Copt{C_{\rm opt}}

\def\Clin{C_{\rm lin}}

\def\osc{{\rm osc}}

\def\Cinv{C_{\rm inv}}
\def\Cenorm{C_{|\!|\!|}}

\def\Cmon{C_{\rm mon}}

\newcounter{statement}
\newenvironment{statement}[2][!]{%
\vskip3mm
\hrule
\hrule
\hrule
\vskip1mm
\noindent%
\refstepcounter{statement}%
\bf#2~\thestatement%
\ifthenelse{\equal{#1}{!}}{.\ }{~(#1).\ }%
\it%
}{%
\vskip1mm
\hrule
\hrule
\hrule
\vskip2mm
}

\newenvironment{theorem}[1][!]{\begin{statement}[#1]{Theorem}}{\end{statement}}
\newenvironment{lemma}[1][!]{\begin{statement}[#1]{Lemma}}{\end{statement}}
\newenvironment{proposition}[1][!]{\begin{statement}[#1]{Proposition}}{\end{statement}}

\newenvironment{remark}[1][!]{\begin{statement}[#1]{Remark}}{\end{statement}}
\newenvironment{algorithm}[1][!]{\begin{statement}[#1]{Algorithm}}{\end{statement}}

\numberwithin{statement}{section}

\makeatletter
\def\@seccntformat#1{\hspace*{4mm}%
  \protect\textup{\protect\@secnumfont
    \ifnum\pdfstrcmp{subsection}{#1}=0 \bfseries\fi
    \csname the#1\endcsname
    \protect\@secnumpunct
  }%
}
\makeatother

\def\a{\boldsymbol{\alpha}}
\def\b{\beta}
\def\h{\hslash}
\def\eps{\varepsilon}
\def\Pe{{\rm Pe}}
\def\jump#1{[\![#1]\!]}

\def\PP{\mathcal{P}}

\def\Ceff{C_{\rm eff}}%
\def\Enorm#1{\norm{#1}{}}%
\def\osc{{\rm osc}}%
\def\supp{{\rm supp}}%
\def\Cdrel{C_{\rm drl}}%
\def\Csz{C_{\rm sz}}%
\def\Ctrace{C_{\rm trace}}%
\def\SS{\mathcal{S}}%

\def\Ccls{C_{\rm cls}}%

\usepackage{fancyhdr}
\advance\footskip0.4cm
\advance\textheight-0.4cm
\calclayout
\title{Optimal Adaptivity for the\\SUPG Finite Element Method}

\author{Christoph Erath}
\author{Dirk Praetorius}

\thanks{C. Erath (corresponding author): TU Darmstadt, Germany; Erath@mathematik.tu-darmstadt.de}
\thanks{D. Praetorius: TU Wien, Austria; Dirk.Praetorius@asc.tuwien.ac.at}

\date{\today}
\begin{document}
 
\begin{abstract}
For convection dominated problems, the \emph{streamline upwind Petrov--Galerkin method} (SUPG),
also named \emph{streamline diffusion finite element method} (SDFEM),
ensures a stable finite element solution.
Based on robust {\sl a~posteriori} error estimators, 
we propose an adaptive mesh-refining algorithm for SUPG and prove that the 
generated SUPG solutions converge at asymptotically optimal rates towards the exact solution.
\end{abstract}


 
\maketitle

\section{Introduction}

\subsection{Model problem}
The numerical approximation of convection-dominated convec\-tion-diffusion equations
is a non trivial task. 
It is well known that the weak solution of such problems exhibits layers, and the 
standard finite element discretization (FEM) leads to oscillations in the FEM solutions,
if these layers are not resolved by the triangulation. 
The work~\cite{Augustin:2011-1} provides a competitive study   
of finite volume based (FVM) and FEM based stabilized discretizations for 
convection-dominated convection-diffusion equations.
They draw the conclusion that FVM is compulsory if solutions without spurious oscillation are needed
and local flux conservation is important.
However, if sharpness and the position of layers are important and the application can tolerate 
small spurious oscillations, the so called 
\emph{Streamline-Upwind-Petrov-Galerkin (SUPG)} method, introduced in~\cite{supg},
is a good choice. In the comparison, SUPG outperforms all other methods
in terms of quality of the approximation versus computing time. Furthermore, one can
easily add the SUPG stabilization to an existing FEM Code.

The design of \emph{a~posteriori estimators} for such problems is, in some sense, 
even more delicate, since the involved constants in the estimates 
should be \emph{robust} with respect
to the variation of the model data.
In~\cite{John:2013-1}, 
a robust residual {\sl a~posteriori} estimator, which estimates 
the error in the natural SUPG norm, is proposed. However, the upper bound relies on
some hypotheses. In~\cite{verfuerth05}, a classical residual based estimator is shown to
be fully robust in the energy norm plus a (non-computable) dual norm. It is interesting
to note that the upper bound of the residual estimator is robust with respect to the energy norm,
whereas the lower bound is only semi-robust, i.e., the constant additionally depends
on the (local) P\'eclet number. 
For more details, we refer to~\cite{verfuerth05} as well as to 
the monograph~\cite{verfuerth}.

As model problem, we consider the following stationary convection-diffusion equation
\begin{subequations}\label{eq:strongform}
\begin{align}
 -\eps\Delta u + \a\cdot\nabla u + \b u &=f\quad\text{in }\Omega,\\
 u &= 0\quad\text{on }\Gamma_D \subseteq\partial\Omega,\\
 \eps\,\frac{\partial u}{\partial\nf} &= g \quad\text{on }\Gamma_N := \partial\Omega\backslash\overline\Gamma_D,
\end{align}
\end{subequations}
where $\Omega \subset \R^d$, $d\ge 2$, is a bounded Lipschitz domain with 
polygonal boundary $\partial\Omega$, and $\nf$ is the normal vector on $\partial \Omega$
pointing outward with respect to $\Omega$. 
As in~\cite[Section~2]{tv}, we assume that the data satisfy the following conditions:
\begin{itemize}
\item $0<\eps\ll1$;
\item $\a\in W^{1,\infty}(\Omega)^d$, $\b\in L^\infty(\Omega)$, $f\in L^2(\Omega)$, $g\in L^2(\Gamma_N)$;
\item $-\frac12\nabla\cdot\a+\b \ge \gamma$ and $\norm{\b}{L^\infty(\Omega)}\le c_\b\gamma$ with some $\eps$-independent constants $c_\b,\gamma\ge0$;
\item the Dirichlet boundary $\Gamma_D$ has positive surface measure and contains the inflow boundary, 
i.e., $\set{x\in\Gamma}{\a(x)\cdot \nf(x)<0}\subseteq\Gamma_D$.
\end{itemize}
Define the bilinear form $b(\cdot,\cdot)$ and the continuous linear functional $F(\cdot)$ by
\begin{align}\label{eq:weakform0}
b(w,v) := \int_\Omega\eps\nabla w\cdot\nabla v + (\a\cdot\nabla w + \b w)\,v\,dx
\quad\text{and}\quad
F(v) := \int_\Omega fv\,dx+\int_{\Gamma_N}gv\,ds.
\end{align}
Then, the weak formulation of~\eqref{eq:strongform} reads as follows: Find $u\in H^1_D(\Omega)$ such that
\begin{align}\label{eq:weakform}
 b(u,v) = F(v)
 \quad\text{for all }v\in H^1_D(\Omega).
\end{align}
We note that the assumptions on the data guarantee that the bilinear form $b(\cdot,\cdot)$ is continuous and elliptic (see, e.g., Lemma~\ref{lemma:enorm} below). In particular, the Lax--Milgram lemma proves that~\eqref{eq:weakform} admits a unique solution $u\in H^1_D(\Omega)$. 

\subsection{Streamline upwind Petrov--Galerkin method (SUPG)}
For $\TT_\bullet$ being a regular triangulation of $\Omega$, 
let $\XX_\bullet:=\PP^p(\TT_\bullet)\cap H^1_D(\Omega)$ be the associated conforming finite element space, which consists of $\TT_\bullet$-piecewise polynomials of degree $p\ge1$. As in the continuous case, the Lax--Milgram lemma applies and proves existence and uniqueness of the Galerkin solution $u_\bullet^\star\in\XX_\bullet$ to
\begin{align}\label{eq:galerkin}
 b(u_\bullet^\star,v_\bullet) =  F(v_\bullet)
 \quad\text{for all }v_\bullet\in\XX_\bullet.
\end{align}
To circumvent instabilities of the Galerkin solution $u_\bullet^\star\in\XX_\bullet$
in case of $0<\eps\ll1$, we use the following stabilized Galerkin formulation~\cite{supg}: Find $u_\bullet\in\XX_\bullet$ such that
\begin{subequations}\label{eq:supg}
\begin{align}
b(u_\bullet,v_\bullet) + \sigma_\bullet(u_\bullet,v_\bullet) 
= F(v_\bullet)
 \quad\text{for all }v_\bullet\in\XX_\bullet,
\end{align}
where $\sigma_\bullet(\cdot,\cdot)$ denotes the stabilization term
\begin{align}
 \sigma_\bullet(u_\bullet,v_\bullet) 
 &= \sum_{T\in\TT_\bullet} \vartheta_T\,\int_T
 \big( -\eps\Delta u_\bullet + \a\cdot\nabla u_\bullet + \b u_\bullet - f\big)\,\a\cdot\nabla v_\bullet\,dx,
\end{align}
\end{subequations}
i.e., we add extra diffusion in the streamline direction.
This explains the name \emph{Streamline-Upwind-Petrov-Galerkin (SUPG)}.
The user-chosen stabilization parameter $\vartheta_T>0$ is discussed in
Section~\ref{section:modelproblem} below. We stress that $\sigma_\bullet(\cdot,\cdot)$ 
is linear in the second variable, but affine in the first. The choice
$\vartheta_T=0$ for all $T\in\TT_\bullet$ leads 
to the standard Galerkin discretization~\eqref{eq:galerkin}.

\subsection{Contributions and outline}
In this work, we study the convergence behavior of an adaptive SUPG algorithm, 
which is steered by the residual error estimator from~\cite{verfuerth05}.
Theorem~\ref{theorem:convergence} and Theorem~\ref{theorem:rates} prove that our adaptive algorithm (Algorithm~\ref{algorithm}) guarantees linear convergence with
asymptotically optimal convergence behavior 
with respect to the number of degrees of freedom.
This result is based on the following three properties of the
{\sl a~posteriori} error estimator, which are called \emph{axioms of adaptivity} in~\cite{axioms}: 
\emph{Stability and reduction} of the estimator, 
stated in Lemma~\ref{lemma:A1-A2}, follow basically from the literature. 
The proof is adapted to our model problem.
\emph{Discrete reliability} is given in Proposition~\ref{proposition:drel} and relies
on a Scott-Zhang-type projector. Furthermore, we have 
to overcome the lack of the classical Galerkin orthogonality
due to the SUPG stabilization.
We remark that our convergence analysis is done in the energy norm.
We did not succeed to handle the non-local dual norm for the 
convergence analysis of the adaptive algorithm because it relies on some local calculations. 
But since the error estimator for the robust and semi-robust
estimates are exactly the same, i.e., the non-local and non-computable dual norm 
absorbs basically the convection terms, we do not see strong restrictions on that. 
Note that the convergence analysis is in fact a statement for the estimator.

We organize the content of the paper as follows.
Section $2$ introduces some mesh quantities and gives a short summary on the most important
facts of the SUPG method.
In section $3$, we introduce the {\sl a posteriori} error estimator and prove 
\emph{stability}, \emph{reduction}, and \emph{discrete reliability}.
We state the adaptive algorithm in section $4$, prove its linear convergence and show that 
our strategy leads to optimal convergence rates.
Numerical experiments in section $5$ confirm our theoretical results. We close the paper with
some conclusions in section $6$.

\bigskip

{\bf General notation.}\quad
We use $\lesssim$ to abbreviate $\le$ up to some (generic) multiplicative constant, which is clear from the context. Moreover, $\simeq$ abbreviates that both estimates $\lesssim$ and $\gtrsim$ hold. Throughout, the mesh-dependence of (discrete) quantities is explicitly stated by use of appropriate indices, e.g., $u_\bullet\in\XX_\bullet$ is the SUPG solution of~\eqref{eq:supg} 
for the triangulation $\TT_\bullet$, and $\eta_\ell(u_\ell^\star)$ is the error estimator with respect to the triangulation $\TT_\ell$ evaluated at the Galerkin approximation $u_\ell^\star\in\XX_\ell$
of~\eqref{eq:galerkin}. 
Finally, in the case $\gamma=0$, we interpret $\gamma^{-s}=\infty$ for all $s>0$.

\section{Preliminaries}
\label{section:modelproblem}

The purpose of this section is to collect the necessary notation as well as some well-known facts on SUPG.

\subsection{Energy norm}
Our analysis employs the operator induced energy norm $\enorm\cdot$, which is an equivalent norm on $H^1_D(\Omega)$:
For $v\in H^1_D(\Omega)$, define
\begin{align}\label{eq:norm}
 \enorm{v}:=\enorm{v}_\Omega,
 \quad\text{where}\quad
 \enorm{v}_\omega := \big(\eps\,\norm{\nabla v}{L^2(\omega)}^2 + \gamma\,\norm{v}{L^2(\omega)}^2\big)^{1/2}
 \quad\text{for all }\omega\subseteq\Omega.
\end{align}
The following elementary lemma (see, e.g.,~\cite[Proposition~4.17]{verfuerth}) states that $b(\cdot,\cdot)$ is elliptic with respect to $\enorm\cdot$ with ellipticity constant $1$.

\begin{lemma}\label{lemma:enorm}
For all $v\in H^1_D(\Omega)$, it holds that $\enorm{v}^2\le b(v,v)$.\qed
\end{lemma}%

\subsection{Mesh-related quantities}
Throughout, we assume that all triangulations $\TT_\bullet$ are conforming and resolve the boundary conditions $\Gamma_D$ and\ $\Gamma_N$, i.e., 
for $\Gamma_\bullet\in\{\Gamma_D,\Gamma_N\}$, it holds that $\overline\Gamma_\bullet = \bigcup\set{T\cap\Gamma}{T\cap\Gamma_\bullet\neq\emptyset}$. The elements 
$T\in\TT_\bullet$ are compact simplices with Euclidean diameter ${\rm diam}(T)$.
The triangulation $\TT_\bullet$ is $\kappa$-shape regular if 
\begin{align}\label{eq:shape_regular}
 \max_{T\in\TT_\bullet}\frac{{\rm diam}(T)}{|T|^{1/d}} \le \kappa < \infty.
\end{align}
Note that this implies equivalence $|T|^{1/d}\le{\rm diam}(T)\le \kappa\,|T|^{1/d}$ for all $T\in\TT_\bullet$. Hence, we will use the element-size
$h_T := |T|^{1/d}\simeq\diam(T)$ throughout the work.
To define robust {\sl a~posteriori} estimators, we define another mesh quantity
\begin{align}\label{eq:hrobust}
 \h_T := \min\{\eps^{-1/2}h_T\,,\,\gamma^{-1/2}\}
 \quad\text{for all $T\in\TT_\bullet$}.
\end{align}
Associated with $\TT_\bullet$, we define the local mesh-size function $h_\bullet\in L^\infty(\Omega)$ by $h_\bullet|_T:=h_T$ for all $T\in\TT_\bullet$.
Finally, we denote by $\TT_\bullet|_{\Gamma_N}$ the induced triangulation of $\Gamma_N$.

\subsection{Mesh-refinement}
We suppose a fixed mesh-refinement strategy such that $\TT_\circ := \refine(\TT_\bullet,\MM_\bullet)$ is the coarsest conforming triangulation such that all marked elements $\MM_\bullet\subseteq\TT_\bullet$ have been refined, i.e., $\MM_\bullet\subseteq\TT_\bullet\backslash\TT_\circ$. We write $\TT_\circ\in\refine(\TT_\bullet)$, if $\TT_\circ$ is obtained by finitely many steps of refinement, i.e., it holds the following: There exists $n\in\N_0$, triangulations $\TT_0,\dots,\TT_n$, and sets of marked elements $\MM_j\subseteq\TT_j$ such that $\TT_0=\TT_\bullet$, $\TT_{j+1}=\refine(\TT_j,\MM_j)$ for all $j=0,\dots,n-1$, and $\TT_\circ=\TT_n$. In particular, it holds that $\TT_\bullet\in\refine(\TT_\bullet)$.

We require the following assumptions on $\refine(\cdot)$: 
\begin{itemize}
\item
First, each element $T\in\TT_\bullet$ is the union of its successors, i.e., $T = \bigcup\set{T'\in\TT_\circ}{T'\subseteq T}$. In particular, this guarantees the pointwise estimate $h_\circ \le h_\bullet$ for the corresponding mesh-size functions. 
\item Second, we suppose that sons of refined elements are uniformly smaller than their fathers, i.e., there exists $0<\qref<1$ such that $h_{T'} \le \qref h_T$ for all $T\in\TT_\bullet\backslash\TT_\circ$ and all $T'\in\TT_\circ$ with $T'\subsetneqq T$. 
\item Finally, we suppose that all meshes $\TT_\circ\in\refine(\TT_\bullet)$ are $\kappa$-shape regular~\eqref{eq:shape_regular}, where $\kappa>0$ depends only on $\TT_\bullet$.
\end{itemize}
Possible choices for $\refine(\cdot)$ include red-green-blue refinement in 2D (see, e.g.,~\cite{verfuerth}) or newest vertex bisection in $\R^d$ for $d\ge2$ (see, e.g.,~\cite{stevenson:nvb,kpp13}). In either case, it holds that $\qref=2^{-1/d}$.

To abbreviate notation and in view of Algorithm~\ref{algorithm} below, let $\TT_0$ be a given conforming triangulation of $\Omega$. Let $\T:=\refine(\TT_0)$ 
be the set of all conforming triangulations, which can be obtained by the fixed refinement strategy.

\subsection{Well-posedness of SUPG}
In order to state the well-posedness of the SUPG formulation~\eqref{eq:supg}, we recall the following inverse estimate (see, e.g.,~\cite[Section~4.5]{brennerscott}).

\begin{lemma}\label{lemma:invest}
There exists a constant $\Cinv>0$, which depends only on $\kappa$-shape regularity of $\TT_\bullet\in\T$ and the polynomial degree $p\ge1$, such that
\begin{align}\label{eq:invest}
 h_T\,\norm{\nabla v_\bullet}{L^2(T)}\le \Cinv\,\norm{v_\bullet}{L^2(T)}
 \quad\text{for all }v_\bullet\in\XX_\bullet
 \text{ and all }T\in\TT_\bullet.
\end{align}
\vspace*{-12mm}\newline\qed\medskip
\end{lemma}%

Note that the SUPG formulation~\eqref{eq:supg} can be recast as
\begin{subequations}\label{eq':supg}
\begin{align}
 b_\bullet(u_\bullet,v_\bullet) =  F_\bullet(v_\bullet)
 \quad\text{for all }v_\bullet\in\XX_\bullet,
\end{align}
where
\begin{align}
 b_\bullet(u_\bullet,v_\bullet) &:= b(u_\bullet,v_\bullet) + \sum_{T\in\TT_\bullet} \vartheta_T\,\int_T
 \big( -\eps\Delta u_\bullet + \a\cdot\nabla u_\bullet + \b u_\bullet\big)\,\a\cdot\nabla v_\bullet\,dx,\\
 F_\bullet(v_\bullet) &:= F(v_\bullet)
 + \sum_{T\in\TT_\bullet} \vartheta_T\,\int_Tf\,\a\cdot\nabla v_\bullet\,dx.
\end{align}
\end{subequations}
For given $\TT_\bullet$, we define the discrete SUPG norm
\begin{align}\label{eq:supgnorm}
 \enorm{v_\bullet}_{\bullet,supg}^2 
 := \enorm{v_\bullet}^2 + \sum_{T\in\TT_\bullet}\vartheta_T\,\norm{\a\cdot\nabla v_\bullet}{L^2(T)}^2.
\end{align}
The following lemma (see, e.g.,~\cite[Lemma~3.25, Remark 3.29]{rst}) proves that SUPG~\eqref{eq:supg} has some enhanced stability compared to the standard Galerkin discretization~\eqref{eq:galerkin}.

\begin{lemma}\label{lemma:supg:elliptic}
Suppose that $\gamma>0$. Let $\Cinv$ be the constant from the inverse estimate of Lemma~\ref{lemma:invest}. 
Let $\TT_\bullet\in\T$.
Suppose that the SUPG stabilization parameters $\vartheta_T$ in~\eqref{eq:supg} satisfy
\begin{align}
 \label{eq:supgstabilconst}
 0 \, < \, \vartheta_T \le \begin{cases}
  \displaystyle\frac12\,\min\Big\{\frac{h_T^2}{\eps\Cinv^2}\,,\,\frac{\gamma}{\norm{\b}{L^\infty(T)}^2}\Big\}
 \quad&\text{for }p\ge2,\\
 \displaystyle\,\frac{\gamma}{\norm{\b}{L^\infty(T)}^2}\quad&\text{for }p=1,
 \end{cases}
\end{align}
for all $T\in\TT_\bullet$.
Then the SUPG bilinear form $b_\bullet(\cdot,\cdot)$ satisfies that
\begin{align}
 \label{eq:supgstabil}
 \frac12\,\enorm{v_\bullet}_{\bullet,supg}^2 \le b_\bullet(v_\bullet,v_\bullet)
 \quad\text{for all $v_\bullet\in\XX_\bullet$.}
\end{align}
In particular, the Lax--Milgram lemma proves that the SUPG formulation~\eqref{eq:supg} admits a unique solution $u_\bullet\in\XX_\bullet$.
\qed
\end{lemma} 

\begin{remark}
 \label{rem:apriori}
 Lemma~\ref{lemma:supg:elliptic} is formulated for $\gamma>0$. 
 If $\gamma=0$ (and hence $\b=0$ by virtue of the assumption $\norm{\b}{L^\infty(\Omega)}\le c_\b\gamma$), 
 suppose in addition that $-\frac12\nabla\cdot\a =0$. Provided that
 \begin{align}\label{eq':supgstabilconst}
 \begin{cases}
 0\,<\, \vartheta_T \le \displaystyle\frac12\,\frac{h_T^2}{\eps\Cinv^2}\quad&\text{for $p\ge2$},\\
 0\,<\,\vartheta_T<\infty&\text{for $p=1$},
 \end{cases}
 \end{align}
 for all $T\in\TT_\bullet$, 
 the proof of~\cite[Lemma~3.25]{rst} remains valid and verifies~\eqref{eq:supgstabil}. 
 In particular, the SUPG formulation~\eqref{eq:supg} admits a unique solution $u_\bullet\in\XX_\bullet$.
 \end{remark}
 
In order to get an optimal {\sl a~priori} estimate for the SUPG error on uniform meshes~\cite[Theorem~3.27 and eq.~(3.38)]{rst}, the stability parameter $\vartheta_T$ is usually chosen as
\begin{align}
 \label{eq:supgstabilconstopt}
 \vartheta_T &:= \begin{cases}
  \delta_0 \, h_T\quad&\text{if }\Pe_T > 1\quad\text{(convection-dominated case)},\\
  \delta_1 h_T^2/\eps\quad&\text{if }\Pe_T \le 1\quad\text{(diffusion-dominated case)}.
 \end{cases}
\end{align}
This also holds if $\gamma=0$.
Here, $\delta_0,\delta_1>0$ are given constants and $\Pe_T$ denotes the local P\'eclet number
\begin{align}
 \label{eq:localpeclet}
 \Pe_T := \frac{\norm{\a}{L^\infty(T)}h_T}{2\eps}.
\end{align}


\subsection{Choice of SUPG stabilization parameter}
In this work, we choose $\vartheta_T$ by~\eqref{eq:supgstabilconstopt}
and suppose that the choice of $\delta_0,\delta_1>0$ 
guarantees~\eqref{eq:supgstabilconst} or~\eqref{eq':supgstabilconst} 
to ensure  existence and uniqueness of the SUPG approximation $u_\bullet\in\XX_\bullet$.
Associated with the triangulation $\TT_\bullet$ and in analogy to $h_\bullet\in L^\infty(\Omega)$, 
we define $\vartheta_\bullet\in L^\infty(\Omega)$ by $\vartheta_\bullet|_T := \vartheta_T$.

\section{A~posteriori error estimation}
\label{section:aposteriori}

\subsection{(Semi-) Robust \textsl{a~posteriori} error control}
Our model problem~\eqref{eq:strongform} allows, in particular, the case of dominating convection or reaction. Therefore, \emph{robust} {\sl a~posteriori} estimates are preferred, i.e., the reliability as well as the efficiency constant 
do not depend on the variation of the model data $\eps$, $\a$, and $\b$, but may depend on $\gamma$ and $c_\beta$. 

In the following, we recall the residual {\sl a~posteriori} error estimator from~\cite{verfuerth05,tv}. 
For $\TT_\bullet\in\T$ and all discrete functions $w_\bullet\in\XX_\bullet$, we define
\begin{subequations}\label{eq:eta}
\begin{align}
 \eta_\bullet(w_\bullet) := \eta_\bullet(\TT_\bullet,w_\bullet)
 \quad\text{with}\quad
 \eta_\bullet(\UU_\bullet,w_\bullet)
 = \Big(\sum_{T\in\UU_\bullet}\eta_\bullet(T,w_\bullet)^2\Big)^{1/2}
 \quad\text{for all }\UU_\bullet\subseteq\TT_\bullet,
\end{align}
where the local contributions read
\begin{align}
 \label{eq:localcontributions}
 \begin{split}
 \eta_\bullet(T,w_\bullet)^2 := \,&\h_T^2\,\norm{-\eps\Delta w_\bullet + \a\cdot\nabla w_\bullet + \b w_\bullet - f}{L^2(T)}^2 
 \\&
 + \h_T\,\eps^{-1/2}\,\norm{\jump{\eps\nabla w_\bullet}}{L^2(\partial T\cap\Omega)}^2 
 + \h_T\,\eps^{-1/2}\,\norm{g-\eps\partial w_\bullet/\partial\nf}{L^2(\partial T\cap\Gamma_N)}^2.
 \end{split}
\end{align}
\end{subequations}
Here, the normal jump reads
$\jump{\mathbf{g}}|_E:=(\mathbf{g}|_T-\mathbf{g}|_{T'})\cdot\nf$,
where $\mathbf{g}|_T$ denotes the trace of $\mathbf{g}$ from $T$ onto $E$ and 
$\nf$ points from $T$ to $T'$ with $E=T\cap T'$.

We note that (up to the scaling $h_T\simeq \h_T$), the error estimator 
for the SUPG finite element method~\eqref{eq:supg} is the same as for the standard finite element discretization~\eqref{eq:galerkin} with $w_\bullet = u_\bullet$ and $w_\bullet = u_\bullet^\star$, respectively. 
Moreover, in either case, there holds reliability
\begin{align}\label{eq:reliable}
 \enorm{u-u_\bullet} \le \Crel\,\eta_\bullet(u_\bullet)
 \quad\text{and}\quad
 \enorm{u-u_\bullet^\star} \le \Crel^\star\,\eta_\bullet(u_\bullet^\star).
\end{align}
We note that the constants $\Crel,\Crel^\star>0$ are \emph{robust}. For the standard Galerkin formulation~\eqref{eq:galerkin}, details are found in the monograph~\cite{verfuerth}. For SUPG~\eqref{eq:supg}, we refer to~\cite{verfuerth05} and to the recent work~\cite{tv}. 

To formulate the lower bound, let $\a_\bullet\in\PP^{p-1}(\TT_\bullet)^d$, $\b_\bullet,f_\bullet\in\PP^{p-1}(\TT_\bullet)$, and $g_\bullet\in\PP^{p-1}(\TT_\bullet|_{\Gamma_N})$ be polynomial approximations of the given data. Define the data oscillations by
\begin{subequations}\label{eq:osc}
\begin{align}
 \osc_\bullet(w_\bullet) :=  \Big(\sum_{T\in\TT_\bullet}\osc_\bullet(T,w_\bullet)^2\Big)^{1/2}
 \text{ for all }w_\bullet\in\XX_\bullet,
\end{align}
where the local contributions read
\begin{align}
 \label{eq:osc:localcontributions}
 \begin{split}
 \osc_\bullet(T,w_\bullet)^2 := &\h_T^2\,\norm{(\a-\a_\bullet)\cdot\nabla w_\bullet + (\b-\b_\bullet) w_\bullet - (f-f_\bullet)}{L^2(T)}^2 
 \\&
 + \h_T\,\eps^{-1/2}\,\norm{g-g_\bullet}{L^2(\partial T\cap\Gamma_N)}^2.
 \end{split}
\end{align}
\end{subequations}%
Note that the efficiency
\begin{align}\label{eq:efficient}
  \eta_\bullet(u_\bullet) \le \Ceff\,\big(\enorm{u-u_\bullet} + \osc_\bullet(u_\bullet)\big)
 \quad\text{and}\quad
  \eta_\bullet(u_\bullet^\star) \le \Ceff^\star\,\big(\enorm{u-u_\bullet^\star} + \osc_\bullet(u_\bullet^\star)\big)
\end{align}%
is in general not robust with respect to the energy norm~\eqref{eq:norm}, 
and $\Ceff,\Ceff^\star>0$ depend additionally on the P\'eclet number, i.e., the estimate
is semi-robust.
However, the lower \emph{and} upper {\sl a~posteriori} bound become robust 
by extending the energy norm in~\eqref{eq:reliable} and~\eqref{eq:efficient}, i.e., $\enorm\cdot$ is replaced by
\begin{align}\label{eq':enorm}
 \Enorm{v}^2 := \enorm{v}^2 + \enorm{\a\cdot\nabla v}_*^2
 \quad\text{with}\quad
 \enorm{\a\cdot\nabla v}_* := \sup_{w\in H^1_D(\Omega)\backslash\{0\}}\frac{\int_\Omega\a\cdot\nabla v\,w\,dx}{\enorm{w}};
\end{align}%
see, e.g.,~\cite[Theorem 2.8 and Remark 2.9]{tv}. 
\begin{remark}
 \label{rem:robustness}
 The local contributions~\eqref{eq:localcontributions} and therefore 
 the estimator~\eqref{eq:eta} are the same for the semi-robust 
 and the robust estimate. Consequently, the semi-robustness or the robustness
 do not affect the output (meshes) of
 the adaptive Algorithm~\ref{algorithm} (below). 
\end{remark}

\subsection{\textsl{A~posteriori} control of SUPG stabilization}
The following lemmas are some key ingredients of~\cite{tv} to prove robust reliability~\eqref{eq:reliable} 
for SUPG~\eqref{eq:supg} in the extended energy norm~\eqref{eq':enorm}. We will use them in a different
context below.

\begin{lemma}\label{lemma':invest}
Let $\TT_\bullet\in\T$.
With the constant $\Cinv$ from~\eqref{eq:invest}, it holds that
\begin{align}\label{eq':hT}
 \frac{h_T}{\h_T}\,\norm{\nabla v_\bullet}{L^2(T)} \le \Cenorm\,\enorm{v_\bullet}_T
 \quad\text{for all }v_\bullet\in\XX_\bullet
 \text{ and all }T\in\TT_\bullet, 
\end{align}
where $\Cenorm=1$ if $\eps^{-1/2}h_T\le \gamma^{-1/2}$ and $\Cenorm=\Cinv$ otherwise.
\end{lemma}%

\begin{proof}
If $\eps^{-1/2}h_T\le \gamma^{-1/2}$, it follows that $\h_T=\eps^{-1/2}h_T$ and hence that
\begin{align*}
 \frac{h_T}{\h_T}\,\norm{\nabla v_\bullet}{L^2(T)} \reff{eq:hrobust}= \eps^{1/2}\,\norm{\nabla v_\bullet}{L^2(T)} \reff{eq:norm}\le \enorm{v_\bullet}_T.
\end{align*}
If $\eps^{-1/2}h_T > \gamma^{-1/2}$, it follows that $\h_T = \gamma^{-1/2}$ and hence that
\begin{align*}
 \frac{h_T}{\h_T}\,\norm{\nabla v_\bullet}{L^2(T)} \reff{eq:hrobust}= \gamma^{1/2}\,h_T\,\norm{\nabla v_\bullet}{L^2(T)} 
 \reff{eq:invest}\le\Cinv\gamma^{1/2}\norm{v_\bullet}{L^2(T)}
 \reff{eq:norm}\le \Cinv\,\enorm{v_\bullet}_T.
\end{align*}
This concludes the proof.
\end{proof}%

\begin{lemma}\label{lemma:sigma-vs-eta}
Recall the constant $\Cenorm$ from Lemma~\ref{lemma':invest}.
For all $\TT_\bullet\in\T$, it holds that
\begin{align}\label{eq:sigma-vs-eta}
 |\sigma_\bullet(w_\bullet,v_\bullet)|
 \le \Cenorm\,\norm{|\a|\,\vartheta_\bullet/h_\bullet}{L^\infty(\bigcup\UU_\bullet)}\,  \eta_\bullet(\UU_\bullet, w_\bullet)\,\enorm{v_\bullet}_{\bigcup\UU_\bullet}
 \quad\text{for all $v_\bullet,w_\bullet\in\XX_\bullet$},
\end{align}
where $\UU_\bullet\subseteq\TT_\bullet$ satisfies 
$\supp(v_\bullet)\subseteq\bigcup\UU_\bullet$.
Moreover, it holds that
\begin{align}\label{aux:constant}
 \norm{|\a|\,\vartheta_\bullet/h_\bullet}{L^\infty(\Omega)} \le \max\{\delta_0\,\norm{\a}{L^\infty(\Omega)}\,,\,2\delta_1\}.
\end{align}
\end{lemma}

\begin{proof}
For $T\in\TT_\bullet$, the definition of the residual error estimator $\eta_\bullet(w_\bullet)$ in~\eqref{eq:eta} gives
\begin{align*}
 \vartheta_T\,\int_T
 \big( -\eps\Delta w_\bullet + \a\cdot\nabla w_\bullet + \b w_\bullet - f\big)\,\a\cdot\nabla v_\bullet\,dx
 \le \frac{\vartheta_T}{\h_T}\,\eta_\bullet(T,w_\bullet)\,\norm{\a}{L^\infty(T)}\,\norm{\nabla v_\bullet}{L^2(T)}.
\end{align*}
Note that the left-hand side vanishes if $T\not\in\UU_\bullet$.
Hence, the stabilization $\sigma_\bullet(w_\bullet,v_\bullet)$ from~\eqref{eq:supg} is bounded by
\begin{align*}
 \sigma_\bullet(w_\bullet,v_\bullet)
 &\le \norm{|\a|\,\vartheta_\bullet/h_\bullet}{L^\infty(\bigcup\UU_\bullet)} \,\sum_{T\in\UU_\bullet}\eta_\bullet(T,w_\bullet)\,\frac{h_T}{\h_T}\,\norm{\nabla v_\bullet}{L^2(T)}
 \\&
 \reff{eq':hT}\le\Cenorm\,\norm{|\a|\,\vartheta_\bullet/h_\bullet}{L^\infty(\bigcup\UU_\bullet)}\,
 \eta_\bullet(\UU_\bullet,w_\bullet)\,\enorm{v_\bullet}_{\bigcup\UU_\bullet}.
\end{align*}
This proves~\eqref{eq:sigma-vs-eta}. To see the boundedness of $\norm{|\a|\,\vartheta_\bullet/h_\bullet}{L^\infty(\Omega)}$, let $T\in\TT_\bullet$. Recall the choice~\eqref{eq:supgstabilconstopt} of $\vartheta_T$.
If $\Pe_T>1$, it holds that $$\norm{\a}{L^\infty(T)}\,\vartheta_T/h_T = \delta_0\,\norm{\a}{L^\infty(T)}.$$ If $\Pe_T\le1$, it holds that $$\norm{\a}{L^\infty(T)}\,\vartheta_T/h_T = \delta_1\,\norm{\a}{L^\infty(T)}\,h_T/\eps = 2\delta_1\,\Pe_T\le2\delta_1.$$ Hence, we conclude the proof.
\end{proof} 

\subsection{Stability and reduction of \textsl{a~posteriori} error estimator}
Following the lines of~\cite[Section~3.1]{ckns}, 
the residual error estimators satisfies \emph{stability on non-refined element domains}~\eqref{axiom:stability} 
as well as \emph{reduction on refined element domains}~\eqref{axiom:reduction}. 
For convenience of the reader, we include the proof adapted to our model problem and error estimator.

\begin{lemma}\label{lemma:A1-A2}
There exist constants $\Cstab>0$ and $0<\qred<1$ such that for all triangulations $\TT_\bullet\in\T$ and $\TT_\circ\in\refine(\TT_\bullet)$ as well as arbitrary discrete functions $v_\bullet\in\XX_\bullet$ and $v_\circ\in\XX_\circ$, the following estimates hold:
\begin{align}\label{axiom:stability}
|\eta_\circ(\UU_\bullet,v_\circ)-\eta_\bullet(\UU_\bullet,v_\bullet)| \le \Cstab\,\enorm{v_\bullet-v_\circ} \quad\text{for all }\,\UU_\bullet\subseteq\TT_\bullet\cap\TT_\circ
\end{align}
as well as
\begin{align}\label{axiom:reduction0}
 \eta_\circ(\UU_\circ,v_\bullet)
 \le \eta_\bullet(\UU_\bullet,v_\bullet)
 \quad\text{for all }\,\UU_\circ\subseteq\TT_\circ
 \text{ and }\, \UU_\bullet := \set{T\in\TT_\bullet}{\exists T'\in\UU_\circ\quad T'\subseteq T}.
\end{align}
Moreover, if $\norm{h_\bullet}{L^\infty(\Omega)}\le (\eps/\gamma)^{1/2}$, it even holds that
\begin{align}\label{axiom:reduction}
 \eta_\circ(\TT_\circ\backslash\TT_\bullet,v_\bullet)
 \le \qred \,\eta_\bullet(\TT_\bullet\backslash\TT_\circ,v_\bullet).
\end{align}
While $\qred=\qref^{1/2}$ depends only on the mesh-refinement (and $\qred = 2^{-1/(2d)}$ for newest vertex bisection), $\Cstab$ depends on shape regularity of $\TT_\circ$, $c_\b$, $\Cinv$, and on the local P\'eclet numbers $\Pe_T$ for all $T\in\TT_\circ$.
\end{lemma}%

\begin{proof}
To prove stability~\eqref{axiom:stability}, we consider the contributions of the error estimator separately. To that end, let $w_\circ := v_\circ - v_\bullet\in\XX_\circ\supseteq\XX_\bullet$ and $T\in\TT_\circ$.
The definition of~$\h_T$ in~\eqref{eq:hrobust} implies that $\h_T\eps^{1/2}/h_T \le 1$. 
It follows
from the inverse estimate~\eqref{eq:invest} that
\begin{align*}
 \h_T\,\norm{-\eps\Delta w_\circ}{L^2(T)}
 = \frac{\h_T\eps^{1/2}}{h_T}\eps^{1/2}h_T\,\norm{\Delta  w_\circ}{L^2(T)}
 \reff{eq:invest}\lesssim \eps^{1/2}\,\norm{\nabla w_\circ}{L^2(T)} \le \enorm{w_\circ}_T.
\end{align*}
The definition of the local P\'eclet number~\eqref{eq:localpeclet} implies that
\begin{align}\label{est:local_peclet}
\begin{split}
 \h_T\,\norm{\a\cdot\nabla w_\circ}{L^2(T)} 
 \le \eps^{-1/2}\,h_T\,\norm{\a}{L^\infty(T)}\,\norm{\nabla w_\circ}{L^2(T)}
 &\reff{eq:localpeclet}= 2\,\Pe_T\,\eps^{1/2}\,\norm{\nabla w_\circ}{L^2(T)}
 \\&
 \reff{eq:norm}\le 2\,\Pe_T\,\enorm{w_\circ}_T.
\end{split}
\end{align}
With $\norm{\b}{L^\infty(T)}\le c_\b\,\gamma$ and $\h_T\le\gamma^{-1/2}$, we obtain that
\begin{align*}
 \h_T\,\norm{\b w_\circ}{L^2(T)} 
 \le \h_T\,\norm{\b}{L^\infty(T)}\norm{w_\circ}{L^2(T)}
 \lesssim \gamma^{1/2}\,\norm{w_\circ}{L^2(T)}
 \reff{eq:norm}\le \enorm{w_\circ}_T.
\end{align*}
Finally, a scaling argument proves $h_T\,\norm{\partial w_\circ/\partial\nf}{L^2(\partial T)}^2 
\lesssim \norm{\nabla w_\circ}{L^2(T)}^2$ and hence
\begin{align*}
 \h_T\eps^{-1/2}\,\norm{\eps\,\partial w_\circ/\partial\nf}{L^2(\partial T)}^2
 = \frac{\h_T\eps^{1/2}}{h_T}\,\eps\,h_T\,\norm{\partial w_\circ/\partial\nf}{L^2(\partial T)}^2
 \lesssim \eps\,\norm{\nabla w_\circ}{L^2(T)}^2 \reff{eq:norm}\le \enorm{w_\circ}_T^2.
\end{align*}
With these four estimates and the inverse triangle inequality, 
the seminorm structure of the error estimator reveals that
\begin{align*}
 &|\eta_\circ(\UU_\bullet,v_\circ)-\eta_\bullet(\UU_\bullet,v_\bullet)| 
 \\&\quad
 \le \Big(\sum_{T\in\UU_\bullet}\h_T^2 \,\norm{-\eps\Delta w_\circ + \a\cdot\nabla w_\circ 
 + \b w_\circ}{L^2(T)}^2 + 2\,\h_T\eps^{-1/2}\norm{\eps\partial w_\circ/\partial \nf}{L^2(\partial T)}^2\Big)^{1/2}
 \\&\quad
 \lesssim \Big(\sum_{T\in\UU_\bullet}\enorm{w_\circ}_T^2\Big)^{1/2} \le \enorm{w_\circ}
 \qquad\text{for all $\UU_\bullet\subseteq\TT_\bullet\cap\TT_\circ$}.
\end{align*}
The hidden constant depends only on shape regularity of $\TT_\circ$, $c_\b$, $\Cinv$, 
and on $\Pe_T$ for all $T\in\TT_\circ$. This proves~\eqref{axiom:stability}.

To prove reduction~\eqref{axiom:reduction}, let $\TT_\circ(T) := \set{T'\in\TT_\circ}{T'\subseteq T}$ 
be the set of sons of $T\in\TT_\bullet\backslash\TT_\circ$. 
Note that the mesh-refinement ensures $h_{T'} \le \qref\,h_T$ 
for some uniform $0<\qref<1$ (and $\qref=2^{-1/d}$ for NVB) as 
well as $T = \bigcup_{T'\in\TT_\circ(T)}T'$. If $\norm{h_\bullet}{L^\infty(\Omega)}\le(\eps/\gamma)^{1/2}$, 
the definition~\eqref{eq:hrobust} implies that $\h_T = \eps^{-1/2}h_T$ 
and hence $\h_{T'} \le \qref\,\h_T$. Otherwise, it holds at least $\h_{T'} \le \h_T$. 
Moreover, it holds that 
$\norm{\jump{\eps \nabla v_\bullet}}{L^2((\partial T'\cap\Omega)\backslash\partial T)} = 0$, 
since $v_\bullet\in\PP^p(T)$ is smooth inside $T\supset T'$. Hence,
\begin{align}\label{aux:reduction}
 \sum_{T'\in\TT_\circ(T)}\eta_\circ(T',v_\bullet)^2
 \le 
 \begin{cases}
  \qref\,\eta_\bullet(T,v_\bullet)^2\quad&\text{if }\norm{h_\bullet}{L^\infty(\Omega)}\le(\eps/\gamma)^{1/2},\\
  \eta_\bullet(T,v_\bullet)^2&\text{otherwise}.
 \end{cases}
\end{align}
Summing this estimate over all $T\in\TT_\bullet\backslash\TT_\circ$, we prove~\eqref{axiom:reduction} with $\qred=\qref^{1/2}$.

To prove monotonicity~\eqref{axiom:reduction0}, note that $\UU_\circ = (\UU_\circ\backslash\TT_\bullet) \cup (\UU_\circ\cap\TT_\bullet)$. From~\eqref{axiom:stability} for $v_\circ=v_\bullet$, we see that
$$
\eta_\circ(\UU_\circ\cap\TT_\bullet,v_\bullet)^2 = \eta_\bullet(\UU_\circ\cap\TT_\bullet,v_\bullet)^2
= \eta_\bullet(\UU_\bullet\cap\TT_\circ,v_\bullet)^2.
$$ 
To treat $\UU_\circ\backslash\TT_\bullet$, we sum~\eqref{aux:reduction} over all $T\in\UU_\bullet\backslash\TT_\circ$ to see that
$$
\eta_\circ(\UU_\circ\backslash\TT_\bullet,v_\bullet)^2
\le \sum_{T\in\UU_\bullet\backslash\TT_\circ}\sum_{T'\in\TT_\circ(T)}\eta_\circ(T',v_\bullet)^2
\reff{aux:reduction}\le \eta_\bullet(\UU_\bullet\backslash\TT_\circ,v_\bullet)^2.
$$
Combining these two estimates with $\UU_\bullet = (\UU_\bullet\cap\TT_\circ)\cup(\UU_\bullet\backslash\TT_\circ)$, we conclude~\eqref{axiom:reduction0}.
\end{proof}

\begin{remark}
In our proof of Lemma~\ref{lemma:A1-A2}, the stability constant $\Cstab$ depends on the local P\'eclet numbers $\Pe_T$ for all $T\in\TT_\circ$. A similar dependence is also observed in the efficiency analysis of the residual error estimator~\eqref{eq:eta} with respect to the energy norm~\eqref{eq:norm}; see~\cite[p.~168, Theorem~4.20]{verfuerth}.
Note that for $\gamma=0$, there is no restriction on the mesh-size in~\eqref{axiom:reduction}.
\end{remark}

\subsection{Discrete reliability}
The following analysis requires a Scott--Zhang-type projector onto $\XX_\bullet = \PP^p(\TT_\bullet) \cap H^1_D(\Omega)$. In addition to the original construction in~\cite{scottzhang}, we also need local 
$L^2$-stability to ensure~\eqref{eq:I_stable}. 
To this end, we recall the construction of~\cite[Section~3.2]{affkp15}.

\begin{lemma}\label{lemma:scottzhang}
For $\TT_\bullet\in\T$, there exists a linear projection $\II_\bullet:L^2(\Omega)\to\XX_\bullet$ with the following properties~{\rm(i)--(v)} for all $T\in\TT_\bullet$ and $\omega_\bullet(T) := \bigcup\set{T'\in\TT_\bullet}{T'\cap T\neq\emptyset}$.
\begin{itemize}
\item[\rm(i)] {\bfseries Projection property:} $\II_\bullet v_\bullet = v_\bullet$ for all $v_\bullet\in\XX_\bullet$.
\item[\rm(ii)] {\bfseries Local definition:} $(\II_\bullet v)|_T$ depends only on $v|_{\omega_\bullet(T)}$ for all $v\in L^2(\Omega)$.
\item[\rm(iii)] {\bfseries Local $\boldsymbol{L^2}$-stability:} $\norm{\II_\bullet v}{L^2(T)} \le \Csz\,\norm{v}{L^2(\omega_\bullet(T))}$ for all $v\in L^2(\Omega)$.
\item[\rm(iv)] {\bfseries Local $\boldsymbol{H^1_D}$-stability:} $\norm{\nabla\II_\bullet v}{L^2(T)} \le \Csz\,\norm{\nabla v}{L^2(\omega_\bullet(T))}$ for all $v\in H^1_D(\Omega)$.
\item[\rm(v)] {\bfseries Approximation property:} $\norm{(1-\II_\bullet) v}{L^2(T)} \le \Csz\,h_T\norm{\nabla v}{L^2(\omega_\bullet(T))}$ for all $v\in H^1_D(\Omega)$.
\end{itemize}
In particular, $\II_\bullet$ satisfies for all $v\in H^1_D(\Omega)$
the following estimates with respect to the energy norm:
\begin{align}\label{eq:I_stable}
 \enorm{\II_\bullet v}_T 
 &\le \Csz\,\enorm{v}_{\omega_\bullet(T)},\\
 \label{eq:I_energynorm1}
\norm{(1-\II_\bullet) v}{L^2(T)} &\le \Csz\,\h_T\,\enorm{v}_{\omega_\bullet(T)},\\ 
 \label{eq:I_energynorm2}
\norm{(1-\II_\bullet) v}{L^2(E)} &\le \Ctrace\,\eps^{-1/4}\h_T^{1/2}\,\enorm{v}_{\omega_\bullet(T)}
\quad\text{for }E\subset \partial T.
\end{align}
The constants $\Csz,\Ctrace>0$ depend only on $\kappa$-shape regularity~\eqref{eq:shape_regular} of $\TT_\bullet$.
\end{lemma}%

\begin{proof}[Sketch of proof]
{\bf Step~1.} The basic construction in~\cite{scottzhang} reads as follows: Let $\{a_i\}$ denote the degrees of freedom of the space $\SS^p(\TT_\bullet) := \PP^p(\TT_\bullet) \cap H^1(\Omega)$. For each $a_i$, one chooses either an element $\sigma_i=T_i$ (if $a_i$ lies inside $T_i$) or a facet $\sigma_i=E_i$ (if $a_i\in E_i\subset\partial T_i$). The nodal value of $(I_\bullet v)(a_i)$ is determined by the values $v|_{\sigma_i}$ of $v$ on $\sigma_i$. In particular, this requires a well-defined trace if $\sigma_i = E_i$, which fails for $L^2$-functions. To extend the construction to $L^2(\Omega)$, we always choose an element $\sigma_i=T_i\ni a_i$. Arguing along the lines of~\cite{scottzhang}, this provides an operator $I_\bullet:L^2(\Omega)\to\SS^p(\TT_\bullet)$, which satisfies~(i)--(v), even for all $v_\bullet\in\SS^p(\TT_\bullet)$ in~(i) and all $v\in H^1(\Omega)$ in (iv)--(v).

{\bf Step~2.} The construction of $I_\bullet$ in Step~1 does not necessarily guarantee that $I_\bullet v\in\XX_\bullet = \PP^p(\TT_\bullet) \cap H^1_D(\Omega)$ even if $v\in H^1_D(\Omega)$. To repair this shortcoming, we define the nodal value $(\II_\bullet v)(a_i) = (I_\bullet v)(a_i)$ if $a_i\in\overline\Omega\backslash\overline\Gamma_D$ and $(\II_\bullet v)(a_i) = 0$ if $a_i\in\overline\Gamma_D$, respectively. 
This ensures $\II_\bullet v\in H^1_D(\Omega)$ and hence $\II_\bullet:L^2(\Omega)\to\XX_\bullet$. 
Arguing along the lines of~\cite{scottzhang}, this provides an operator $\II_\bullet:L^2(\Omega)\to\SS^p(\TT_\bullet)$, which satisfies~(i)--(v).

{\bf Step~3.} To see the stability estimate~\eqref{eq:I_stable}, note that (iii)--(iv) yield that
 $\enorm{\II_\bullet v}_T^2
 = \eps\,\norm{\nabla\II_\bullet v}{L^2(T)}^2
 + \gamma\,\norm{\II_\bullet v}{L^2(T)}^2
 \le\Csz^2\big(\eps\,\norm{\nabla v}{L^2(\omega_\bullet(T))}^2
 + \gamma\,\norm{v}{L^2(\omega_\bullet(T))}^2\big)
 = \enorm{v}_{\omega_\bullet(T)}^2$. 
 Similarly, the estimates~\eqref{eq:I_energynorm1}--\eqref{eq:I_energynorm2} follow from (iii)--(v)
 and the arguments in~\cite[Lemma 3.3]{verfuerth05}.
 This concludes the proof.
\end{proof}

\begin{proposition}[Discrete reliability]\label{proposition:drel}
Let $\TT_\bullet\in\T$ and $\TT_\circ \in \refine(\TT_\bullet)$.
Define $\RR_{\bullet\circ} := \set{T\in\TT_\bullet}{\exists T'\in \TT_\bullet\backslash\TT_\circ\quad T\cap T'\neq\emptyset}$, which consists of refined elements $\TT_\bullet\backslash\TT_\circ$ plus one addition layer of elements.
Then the corresponding SUPG solutions $u_\bullet\in\XX_\bullet$ 
and $u_\circ\in\XX_\circ$ to~\eqref{eq:supg} satisfy that 
\begin{align}\label{eq:reliable2}
 \enorm{u_\circ-u_\bullet}
 \le \Cdrel\,\eta_\bullet(\RR_{\bullet\circ},u_\bullet)
 \quad\text{and}\quad
 \#\RR_{\bullet\circ}\le\Cdrel'\,\#(\TT_\bullet\backslash\TT_\circ).
\end{align}
The constant $\Cdrel'>0$ depends only on $\kappa$-shape regularity~\eqref{eq:shape_regular} of $\TT_\bullet$, 
while $\Cdrel>0$ depends additionally on the constants $\delta_0,\delta_1>0$ from~\eqref{eq:supgstabilconstopt}, and on $\norm{\a}{L^\infty(\Omega)}$.
\end{proposition}

\begin{proof}
 Instead of the classical Galerkin orthogonality of $e_\circ:=u_\circ-u_\bullet$ on $\XX_\circ$,
 the respective SUPG formulations~\eqref{eq:supg} yield that
\begin{align}\label{eq:SUPG_orthogonality}
b(e_\circ,v_\bullet) + \sigma_\circ(u_\circ,v_\bullet)
-\sigma_\bullet(u_\bullet,v_\bullet) 
= 0
 \quad\text{for all }v_\bullet\in\XX_\bullet.
\end{align}
Recall that the SUPG stabilization $\sigma_\bullet(\cdot,\cdot)$ is linear in the second argument, i.e.,
\begin{align}\label{eq:sigma_linear}
\sigma_\circ(u_\circ,e_\circ) -\sigma_\circ(u_\circ,v_\bullet)
= \sigma_\circ(u_\circ,e_\circ-v_\bullet)
\quad\text{for all }v_\bullet\in\XX_\bullet.
\end{align}
For all $v_\bullet\in\XX_\bullet$, it thus follows from the ellipticity~\eqref{eq:supgstabil} of the SUPG bilinear form that
 \begin{align*}
 \frac12\,\enorm{e_\circ}^2
 &\le \frac12\,\enorm{e_\circ}_\circ^2 
 \reff{eq:supgstabil}\le b_\circ(e_\circ,e_\circ) = b(e_\circ,e_\circ)+\sigma_\circ(u_\circ,e_\circ)-\sigma_\circ(u_\bullet,e_\circ)
 \\&
 \reff{eq:SUPG_orthogonality}= b(e_\circ,e_\circ-v_\bullet)
 +\sigma_\circ(u_\circ,e_\circ)-\sigma_\circ(u_\bullet,e_\circ)-\sigma_\circ(u_\circ,v_\bullet)+\sigma_\bullet(u_\bullet,v_\bullet) 
 \\&
 \reff{eq:sigma_linear}= b(u_\circ,e_\circ-v_\bullet)+\sigma_\circ(u_\circ,e_\circ-v_\bullet)
 -b(u_\bullet,e_\circ-v_\bullet) - \sigma_\circ(u_\bullet,e_\circ) +\sigma_\bullet(u_\bullet,v_\bullet)
 \\&\,
 \reff{eq:supg}= \int_\Omega f(e_\circ-v_\bullet)\,dx+\int_{\Gamma_N}g(e_\circ-v_\bullet)\,ds
 -b(u_\bullet,e_\circ-v_\bullet) 
 -\sigma_\circ(u_\bullet,e_\circ)+\sigma_\bullet(u_\bullet,v_\bullet).
 \end{align*}
 As in~\cite{stevenson07}, we choose $v_\bullet := \II_\bullet e_\circ \in \XX_\bullet$.
 Recall the definition of $\RR_{\bullet\circ}$ and note that for $T \in \TT_\bullet\backslash\RR_{\bullet\circ}$, it follows that $\omega_\bullet(T)\subseteq\bigcup(\TT_\bullet\cap\TT_\circ)$. According to Lemma~\ref{lemma:scottzhang}~(i)--(ii), this implies $(e_\circ - v_\bullet)|_T = 0$ for all $T\in\TT_\bullet\backslash\RR_{\bullet\circ}$.
Adapting the (standard) arguments of~\cite[Section~2.3]{tv}
 with~\eqref{eq:I_energynorm1}--\eqref{eq:I_energynorm2}, it follows that
 \begin{align}\label{eq1:drel}
  \enorm{e_\circ}^2\le \Cdrel'\,\eta_\bullet(\RR_{\bullet\circ},u_\bullet)\,\enorm{e_\circ}  
-\sigma_\circ(u_\bullet,e_\circ)+\sigma_\bullet(u_\bullet,v_\bullet),
 \end{align}
 where the constant $\Cdrel'>0$ depends only on shape regularity of $\TT_\bullet$.
Moreover, recall that $v_\bullet=e_\circ$ on all non-refined elements $T\in\TT_\bullet\backslash\RR_{\bullet\circ} \subseteq \TT_\bullet\cap\TT_\circ$. This yields that
 \begin{align*}
  -\sigma_\circ(u_\bullet,e_\circ)+\sigma_\bullet(u_\bullet,v_\bullet)
  =&-\!\!\!\!\sum_{T\in\TT_\circ\backslash(\TT_\bullet\backslash\RR_{\bullet\circ})}\!\!\!\!
  \vartheta_T\,\int_T
   \big( -\eps\Delta u_\bullet + \a\cdot\nabla u_\bullet + \b u_\bullet - f\big)\,\a\cdot\nabla e_\circ\,dx\\
   &+\!\!\!\!\sum_{T\in\TT_\bullet\backslash(\TT_\bullet\backslash\RR_{\bullet\circ})}\!\!\!\!
   \vartheta_T\,\int_T
    \big( -\eps\Delta u_\bullet + \a\cdot\nabla u_\bullet + \b u_\bullet - f\big)\,\a\cdot\nabla v_\bullet\,dx.
 \end{align*}
Note that $\TT_\bullet\backslash(\TT_\bullet\backslash\RR_{\bullet\circ}) = \RR_{\bullet\circ}$. 
Define $\RR_{\circ\bullet} := \set{T\in\TT_\circ}{\exists T'\in\TT_\circ\backslash\TT_\bullet\quad T\cap T'\neq\emptyset}$ and note that $\TT_\circ\backslash(\TT_\bullet\backslash\RR_{\bullet\circ}) = \RR_{\circ\bullet}$. 
Arguing as in the proof of Lemma~\ref{lemma:sigma-vs-eta}, we see that
\begin{align}\label{eq2:drel}
\begin{split}
 &-\sigma_\circ(u_\bullet,e_\circ)+\sigma_\bullet(u_\bullet,v_\bullet)
 \\&\qquad
 \le \Cenorm\,
 \norm{|\a|\,\vartheta_\circ/h_\circ}{L^\infty(\Omega)}\,
  \eta_\circ(\RR_{\circ\bullet}, u_\bullet)\,\enorm{e_\circ}
  + \Cenorm\,\norm{|\a|\,\vartheta_\bullet/h_\bullet}{L^\infty(\Omega)}\,
  \eta_\bullet(\RR_{\bullet\circ},u_\bullet)\,\enorm{v_\bullet}
  \\&\qquad
  \reff{aux:constant}\le \Cenorm\,\max\{\delta_0\,\norm{\a}{L^\infty(\Omega)}\,,\,2\delta_1\}\,
  \big(\eta_\circ(\RR_{\circ\bullet}, u_\bullet)\,\enorm{e_\circ} + \eta_\bullet(\RR_{\bullet\circ},u_\bullet)\,\enorm{v_\bullet}\big).
\end{split}
\end{align}
The stability~\eqref{eq:I_stable} proves that $\enorm{v_\bullet} = \enorm{\II_\bullet e_\circ} \le \Csz\,\enorm{e_\circ}$. The monotonicity~\eqref{axiom:reduction0} proves that $\eta_\circ(\RR_{\circ\bullet}, u_\bullet) \le \eta_\bullet(\RR_{\bullet\circ}, u_\bullet)$. Combining these observations with~\eqref{eq1:drel}--\eqref{eq2:drel}, we obtain
\begin{align*}
 \enorm{e_\circ} \le \eta_\bullet(\RR_{\bullet\circ},u_\bullet)\,
  \big(\Cdrel' + (1+\Csz)\,\Cenorm\,\max\{\delta_0\,\norm{\a}{L^\infty(\Omega)}\,,\,2\delta_1\}\big).
\end{align*}
Altogether, we thus conclude the proof.
\end{proof}

\section{Adaptive mesh-refinement}

\subsection{Adaptive algorithm}
The following algorithm employs the {\sl a~posteriori} error estimator from Section~\ref{section:aposteriori} to steer an adaptive mesh-refinement. We follow an idea of~\cite{bhp17}
and ensure {\sl a~priori} convergence of the adaptive algorithm by refining in each step at least one largest element.

\begin{algorithm}\label{algorithm}
{\sc Input:} Conforming triangulation $\TT_0$, parameters $0<\theta\le1$, $\Cmark\ge1$.\\
{\sc Loop:} For all $\ell=0,1,2,\dots$, iterate the following steps~{\rm(i)--(v)}.
\begin{itemize}
\item[\rm(i)] Compute discrete SUPG solution $u_\ell\in\XX_\ell$ with~\eqref{eq:supg}.
\item[\rm(ii)] Compute refinement indicators $\eta_\ell(T,u_\ell)$ for all $T\in\TT_\ell$ 
with~\eqref{eq:localcontributions}.
\item[{\rm(iii)}] Determine a set $\MM_\ell'\subseteq\TT_\ell$ of up to the multiplicative factor $\Cmark$ minimal cardinality such that $\theta\,\eta_\ell(u_\ell)^2\le \eta_\ell(\MM_\ell',u_\ell)^2$.
\item[{\rm(iv)}] Enlarge $\MM_\ell'$ to $\MM_\ell\supseteq\MM_\ell'$ such that $\#\MM_\ell\le2\,\#\MM_\ell'$ and $\max_{T\in\MM_\ell}h_T = \max_{T\in\TT_\ell}h_T$, i.e., $\MM_\ell$ contains at least one of the largest elements of $\TT_\ell$.
\item[\rm(v)] Refine (at least) the marked elements to obtain the new triangulation $\TT_{\ell+1}:=\refine(\TT_\ell,\MM_\ell)$.
\end{itemize}
{\sc Output:} Discrete solution $u_\ell\in\XX_\ell$ and corresponding {\sl a~posteriori} error estimator $\eta_\ell(u_\ell)$ for all $\ell\in\N_0$.
\end{algorithm}

\begin{lemma}\label{lemma:h}
Algorithm~\ref{algorithm} guarantees $\norm{h_\ell}{L^\infty(\Omega)}\to0$ and, in particular, convergence $\enorm{u-u_\ell}\to0$ as $\ell\to\infty$.
\end{lemma}

\begin{proof}
Convergence $\norm{h_\ell}{L^\infty(\Omega)}\to0$ as $\ell\to\infty$ follows from the fact that each step refines one largest element $T\in\TT_\ell$ and hence ensures the local 
contraction $h_{\ell+1}|_{T'} \le \qref\,h_\ell|_T$ for some fixed $0<\qref<1$
and for all $T'\in\TT_{\ell+1}$ with $T'\subsetneqq T$. To see the second claim, recall the reformulation~\eqref{eq':supg} of SUPG. Note that our assumptions on the mesh-refinement ensures that all meshes $\TT_\ell$ are uniformly $\kappa$-shape regular, where $\kappa>0$ depends only on $\TT_0$.
With the inverse estimate~\eqref{eq:invest} and $\vartheta_T\lesssim h_T^2$ (as $\norm{h_\ell}{L^\infty(\Omega)}\to0$), it follows that 
\begin{align*}
 |b_\ell(u_\ell,v_\ell)-b(u_\ell,v_\ell)|
 &\lesssim \norm{h_\ell}{L^\infty(\Omega)}\,\norm{u_\ell}{H^1(\Omega)}\norm{v_\ell}{H^1(\Omega)},
 \\
 |F_\ell(v_\ell) - F(v_\ell)| 
 &\lesssim \norm{h_\ell}{L^\infty(\Omega)}^2 \,\norm{f}{L^2(\Omega)}\,\norm{v_\ell}{H^1(\Omega)}.
\end{align*}
Hence, the first Strang lemma (e.g.,~\cite[Section III.1.1]{braess}) concludes the proof.
\end{proof}

\def\Caux{C_{\rm aux}}%
\def\Ceq{C_{\rm eq}}%
\subsection{Linear convergence}
In this section, we prove that Algorithm~\ref{algorithm} leads even to linear convergence of the error estimator in the sense of~\cite{axioms}. We note that the proof is \emph{not} robust in the sense of~\cite{verfuerth05}, since it relies on a comparison of SUPG and the standard Galerkin FEM.
More precisely, recall that the constant $\Cstab$ from~\eqref{axiom:stability} is not robust, but depends on the local P\'eclet number.

\begin{lemma}\label{lemma:equivalence}
Recall the constant $\Cenorm$ from Lemma~\ref{lemma:invest}.
With $\Caux:=\Cstab\Cenorm$
it holds that, for all $\ell\in\N_0$ and all $\UU_\ell\subseteq\TT_\ell$,
\begin{align}\label{eq:lemma:equivalence}
\begin{split}
 \max\{\eta_\ell(\UU_\ell,u_\ell),\eta_\ell(\UU_\ell,u_\ell^\star)\}
 &\le \min\{\eta_\ell(\UU_\ell,u_\ell),\eta_\ell(\UU_\ell,u_\ell^\star)\}
 \\&\qquad
 + \Caux\,\norm{|\a|\,\vartheta_\bullet/h_\bullet}{L^\infty(\Omega)}\,\eta_\ell(u_\ell).
\end{split}
\end{align}
\end{lemma}

\begin{proof}
With stability~\eqref{axiom:stability}
for $\TT_\ell=\TT_\bullet=\TT_\circ$, it follows that
\begin{align*}
 \max\{\eta_\ell(\UU_\ell,u_\ell)\,,\,\eta_\ell(\UU_\ell,u_\ell^\star)\}
 \le \min\{\eta_\ell(\UU_\ell,u_\ell)\,,\,\eta_\ell(\UU_\ell,u_\ell^\star)\}
  + \Cstab\,\enorm{u_\ell^\star-u_\ell}.
\end{align*}
With Lemma~\ref{lemma:enorm} and Lemma~\ref{lemma:sigma-vs-eta}, we infer for the last term that
\begin{align*}
 &\enorm{u_\ell^\star-u_\ell}^2
 \le b(u_\ell^\star-u_\ell,u_\ell^\star-u_\ell)
 = b(u_\ell^\star,u_\ell^\star-u_\ell)
 -b(u_\ell,u_\ell^\star-u_\ell)
 =  \sigma_\ell(u_\ell,u_\ell^\star-u_\ell)
 \\&\quad
 \le \Cenorm\,\norm{|\a|\,\vartheta_\ell/h_\ell}{L^\infty(\Omega)}\,  
 \eta_\ell(u_\ell)\,\enorm{u_\ell^\star-u_\ell}.
\end{align*}%
Combining the last two estimates, we conclude the proof.
\end{proof}

\begin{theorem}\label{theorem:convergence}
For all $0<\theta\le 1$, there exists some index $\ell_0\in\N_0$ 
as well as constants $\Clin>0$ and $0<\qlin<1$ such that
\begin{align}\label{eq:convergence}
 \eta_{\ell+k}(u_{\ell+k}) \le \Clin\qlin^k\,\eta_\ell(u_\ell) 
 \quad\text{for all }\ell\ge\ell_0\text{ and all }k\in\N_0.
\end{align}
Moreover, there exists $\Ceq>0$ such that
\begin{align}\label{eq:eta:equivalence}
 \Ceq^{-1}\,\eta_\ell(u_\ell) 
 \le \eta_\ell(u_\ell^\star)
 \le\Ceq\,\eta_\ell(u_\ell) 
 \quad\text{for all }\ell\ge\ell_0.
\end{align}
The constants $\Ceq$, $\Clin$, and $\qlin$ depend only on $\theta$, $\Cstab$, $\qred$, and $\Crel$.
\end{theorem}

\begin{proof}
Recall from Lemma~\ref{lemma:h} that $\norm{h_\ell}{L^\infty(\Omega)}\to0$ as $\ell\to\infty$. By definition~\eqref{eq:supgstabilconstopt} of $\vartheta_T$, there exists an index $\ell_0'\in\N_0$ such that $\norm{|\a|\,\vartheta_\ell/h_\ell}{L^\infty(\Omega)} \simeq \norm{h_\ell}{L^\infty(\Omega)}\le(\eps/\gamma)^{1/2}$ for all $\ell\ge\ell_0'$. In particular, \eqref{axiom:stability}--\eqref{axiom:reduction} are satisfied for all $\ell\ge\ell_0'$ and $\TT_\ell=\TT_\bullet$. Moreover, Lemma~\ref{lemma:equivalence} yields that
\begin{align}
 \label{eq:estFEMSUPG1}
 \eta_\ell(u_\ell^\star)
 \reff{eq:lemma:equivalence}\le \eta_\ell(u_\ell) + C\,\norm{h_\ell}{L^\infty(\Omega)}\,\eta_\ell(u_\ell)
 = (1+C\,\norm{h_\ell}{L^\infty(\Omega)})\,\eta_\ell(u_\ell).
\end{align}
as well as, for all $\UU_\ell\subseteq\TT_\ell$,
\begin{align}
  \label{eq:estFEMSUPG2}
 \eta_\ell(\UU_\ell,u_\ell) \le \eta_\ell(\UU_\ell,u_\ell^\star) + C\,\norm{h_\ell}{L^\infty(\Omega)}\,\eta_\ell(u_\ell).
\end{align}
Since $\norm{h_\ell}{L^\infty(\Omega)}\to0$ as $\ell\to\infty$, there exists an index $\ell_0''\ge\ell_0'$ such that $\sqrt\theta-C\,\norm{h_\ell}{L^\infty(\Omega)}\ge c > 0$ for all $\ell\ge\ell_0''$.
For $\UU_\ell=\TT_\ell$ and since $C\,\norm{h_\ell}{L^\infty(\Omega)}<\sqrt{\theta}\leq 1$, we hence see that
\begin{align*}
 c\,\eta_\ell(u_\ell) 
 \le (1-C\,\norm{h_\ell}{L^\infty(\Omega)})\,\eta_\ell(u_\ell) 
 \stackrel{\rm\eqref{eq:estFEMSUPG2}}{\le} \eta_\ell(u_\ell^\star) 
 \stackrel{\rm\eqref{eq:estFEMSUPG1}}{\le} (1+C\,\norm{h_\ell}{L^\infty(\Omega)}) \, \eta_\ell(u_\ell)
 \le (1+\sqrt\theta)\,\eta_\ell(u_\ell).
\end{align*}
This proves~\eqref{eq:eta:equivalence} with $\Ceq = \max\{1+\sqrt{\theta} \, , \, c^{-1}\}$. To prove~\eqref{eq:convergence}, consider $\UU_\ell=\MM_\ell$.
From the D\"orfler marking in step~(iii) of Algorithm~\ref{algorithm}, we infer that
\begin{align*}
 \frac{\sqrt\theta-C\,\norm{h_\ell}{L^\infty(\Omega)}}{1+C\,\norm{h_\ell}{L^\infty(\Omega)}}\,\eta_\ell(u_\ell^\star)
 &\stackrel{\rm\eqref{eq:estFEMSUPG1}}{\le} \big(\sqrt\theta-C\,\norm{h_\ell}{L^\infty(\Omega)}\big)\eta_\ell(u_\ell)\\
 &\stackrel{\rm(iii)}{\le} \eta_\ell(\MM_\ell,u_\ell) - C\,\norm{h_\ell}{L^\infty(\Omega)}\,\eta_\ell(u_\ell)
 \stackrel{\rm\eqref{eq:estFEMSUPG2}}{\le}\eta_\ell(\MM_\ell,u_\ell^\star).
\end{align*}
Hence, the latter estimate yields that
\begin{align*}
 \theta'\,\eta_\ell(u_\ell^\star)^2 
 := \Big(\frac{c}{1+\sqrt\theta}\Big)^2\,\eta_\ell(u_\ell^\star)^2 
 \le \Big(\frac{\sqrt\theta-C\,\norm{h_\ell}{L^\infty(\Omega)}}{1+C\,\norm{h_\ell}{L^\infty(\Omega)}}\Big)^2\,\eta_\ell(u_\ell^\star)^2
 \le \eta_\ell(\MM_\ell,u_\ell^\star)^2,
\end{align*}
which is the D\"orfler marking with $u_\ell^\star$ for some fixed parameter $0<\theta' < \theta$. Hence,~\cite[Theorem~19]{bhp17} applies and proves that there exists some index $\ell_0\ge\ell_0''$ such that
\begin{align*}
 \eta_{\ell+k}(u_{\ell+k}^\star) \le \Clin'\qlin^k\,\eta_\ell(u_\ell^\star) 
 \quad\text{for all }\ell\ge\ell_0\text{ and all }k\in\N_0,
\end{align*}
where we note that the proof relies only on validity of~\eqref{axiom:stability}--\eqref{axiom:reduction} and on reliability~\eqref{eq:reliable} for the sequence $\TT_\ell$, $\ell\ge\ell_0''$. The constants $\Clin'>0$ and $0<\qlin<1$ depend only on $\theta'$, $\Cstab$, $\qred$, and $\Crel$.
Combining the last estimate with~\eqref{eq:eta:equivalence}, we prove that
$$
\eta_{\ell+k}(u_{\ell+k}) \le \Ceq^2\,\Clin'\qlin^k\,\eta_\ell(u_\ell)
\quad \text{for all $\ell\ge\ell_0$ and all $k\in\N_0$.}
$$
This concludes the proof with $\Clin:=\Ceq^2\,\Clin'$.
\end{proof}

\subsection{Optimal convergence rates}
To show optimal convergence rates, we need to define certain nonlinear approximation classes:
Let $\TT_0$ be the fixed initial triangulation of Algorithm~\ref{algorithm}. Suppose that newest vertex bisection~\cite{stevenson:nvb} is used for $\refine(\cdot)$. Let $\T:=\refine(\TT_0)$ be the set of all triangulations, which can be obtained from $\TT_0$.
For $N>0$, we abbreviate $\T_N:=\set{\TT_\bullet\in\refine(\TT_0)}{\#\TT_\bullet-\#\TT_0\le N}$, where $\#\TT_\bullet$ denotes the number of elements in $\TT_\bullet$.
For all $s>0$, we define the approximability measure
\begin{align}\label{eq:As}
 \norm{u}{\mathbb A_s} 
 := \sup_{N>0} \min_{\substack{\TT_\bullet\in\T_N}} (N+1)^s\,\eta_\bullet(u_\bullet),
\end{align}
where $\eta_\bullet(u_\bullet)$ denotes the residual error estimator~\eqref{eq:eta} associated with the optimal triangulation $\TT_\bullet$. Note that $\norm{u}{\mathbb A_s} < \infty$ means that an algebraic decay $\eta_\bullet = \OO(N^{-s})$ is theoretically possible if for each $N>0$ the optimal triangulation $\TT_\bullet\in\T_N$ is chosen.

The following theorem implies that adaptive SUPG finite elements are optimal in the sense that they converge asymptotically with any possible rate $s>0$.

\begin{theorem}\label{theorem:rates}
Let $0<\theta<\theta_{\rm opt}:=(1+\Cstab^2\Cdrel^2)^{-1}$. Let $\ell_0\in\N_0$ be the index from Theorem~\ref{theorem:convergence}. Then, for all $s>0$, there exists a constant $\Copt>0$ such that the output of Algorithm~\ref{algorithm} satisfies
\begin{align}\label{eq:optimal}
 \norm{u}{\A_s}<\infty
 \quad\Longleftrightarrow\quad
 \forall \ell\ge\ell_0\quad
 \eta_\ell(u_\ell) \le \Copt\,\norm{u}{\A_s}\,(\#\TT_\ell-\#\TT_0+1)^{-s}.
\end{align}
The constant $\Copt>0$ depends only on $\theta$, $\Cmark$, $\TT_0$, $\Clin$, $\qlin$, $\ell_0$, and $s$, as well as on the use of newest vertex bisection.
\end{theorem}

\begin{remark}
We note that the approximability measure~\eqref{eq:As} can also be characterized by means of the total error, i.e., for all $s > 0$, it holds that
\begin{align}
 \norm{u}{\A_s} < \infty
 \,\, \Longleftrightarrow \,\,
 \norm{u}{\mathbb{E}_s}
 := \sup_{N>0} \min_{\substack{\TT_\bullet\in\T_N}} (N+1)^s\,
 \inf_{w_\bullet \in \XX_\bullet} \big( \enorm{u-w_\bullet} + \osc_\bullet(w_\bullet)\big)
 < \infty;
 \end{align}
see~\cite[Section~4.2]{bhp17} for details. Up to data oscillations~\eqref{eq:osc}, Theorem~\ref{theorem:rates} hence proves that Algorithm~\ref{algorithm} leads to asymptotically optimal convergence behaviour for the energy error.
\end{remark}

The proof of Theorem~\ref{theorem:rates} requires some auxiliary results. The simple proof of the following quasi-monotonicity is only included for the convenience of the reader, since the corresponding result~\cite[Section~3.4]{axioms} involves the reduction~\eqref{axiom:reduction}. In the present SUPG setting, this would require that $\TT_\bullet$ is sufficiently fine. However, we stress that monotonicity~\eqref{axiom:reduction0} is sufficient for the proof and hence any further assumption on $\TT_\bullet$ can, in fact, be avoided.

\begin{lemma}
For $\TT_\bullet\in\T_N$ and $\TT_\circ\in\refine(\TT_\bullet)$, it holds that
\begin{align}\label{eq:monotone}
 \eta_\circ(u_\circ)
 \le \Cmon\,\eta_\bullet(u_\bullet),
 \quad\text{with}\quad \Cmon := 1 + \Cstab\Cdrel.
\end{align}
\end{lemma}%

\begin{proof}
With stability~\eqref{axiom:stability}, i.e., apply 
$\UU_\bullet=\TT_\bullet=\TT_\circ$,
 $v_\circ=u_\circ\in\XX_\circ$, $v_\bullet=u_\bullet\in\XX_\bullet\subset\XX_\circ$ , monotonicity~\eqref{axiom:reduction0}, and discrete reliability~\eqref{eq:reliable2}, it follows that
\begin{align*}
 \eta_\circ(u_\circ)
 \le \eta_\circ(u_\bullet) + \Cstab\,\enorm{u_\circ-u_\bullet}
 & 
 \le \eta_\bullet(u_\bullet) + \Cstab\,\enorm{u_\circ-u_\bullet}
 \le \big(1+\Cstab\Cdrel\big)\,\eta_\bullet(u_\bullet).
\end{align*}
This concludes the proof. 
\end{proof}%

The following lemma is a direct consequence of \cite[Proposition~4.12(ii)]{axioms} 
and \cite[Lemma~4.14]{axioms}
(see also~\cite[Section~4.3]{bhp17}). 
As far as the error estimator is concerned, its proof relies therefore only 
on stability~\eqref{axiom:stability}, discrete reliability~\eqref{eq:reliable2}, 
and quasi-monotonicity~\eqref{eq:monotone}.  
We note, however, that the result depends also on the so-called \emph{overlay estimate} of triangulations obtained from newest vertex bisection, i.e., for all $\TT_\bullet,\TT_\star\in\T=\refine(\TT_0)$, there exists a common refinement $\TT_\circ\in\refine(\TT_\bullet)\cap\refine(\TT_\star)$ such that $\#\TT_\circ \le \#\TT_\bullet + \#\TT_\star - \#\TT_0$; see~\cite{stevenson07,ckns}.

\begin{lemma}\label{lemma:comparison}
Let $s>0$ with $\norm{u}{\A_s}<\infty$ and $0<\theta<\theta_{\rm opt}:=(1+\Cstab^2\Cdrel^2)^{-1}$. 
Let $\TT_\bullet\in\T_N$. Then there exists $\TT_\circ\in\refine(\TT_\bullet)$ such that the set $\RR_{\bullet\circ}\subseteq\TT_\bullet$ from Proposition~\ref{proposition:drel} satisfies that
\begin{align}\label{eq:comparison}
 \#\RR_{\bullet\circ} \le C\,\norm{u}{\A_s}^{1/s}\,\eta_\bullet(u_\bullet)^{-1/s}
 \quad\text{and}\quad
 \theta\,\eta_\bullet(u_\bullet)^2 \le \eta_\bullet(\RR_{\bullet\circ},u_\bullet)^2.
\end{align}
The constant $C>0$ depends only on $s>0$, $\theta$, $\Cmon$, $\Cstab$, $\Cdrel$, and on the use of newest vertex bisection.\qed
\end{lemma}

\begin{proof}[Proof of Theorem~\ref{theorem:rates}]
The implication ``$\Longleftarrow$'' follows by definition of the approximation class (cf.~\cite[Proposition~4.15]{axioms}). Therefore, we may focus on the converse implication ``$\Longrightarrow$''. 

Without loss of generality, we may assume $\eta_\ell(u_\ell)>0$ for all $\ell\in\N_0$, since otherwise there exists an index $\ell_1\in\N_0$ with $\MM_\ell=\emptyset$ and the algorithm stagnates with $\TT_\ell = \TT_{\ell_1}$ and $u=u_\ell$ for all $\ell\ge\ell_1$. In particular, it holds that $\MM_\ell\neq\emptyset$ for all $\ell\in\N_0$.
Adopt the notation of Lemma~\ref{lemma:comparison}. 
For $\TT_\bullet = \TT_\ell$, Lemma~\ref{lemma:comparison} provides a set $\RR_\ell \subseteq\TT_\ell$ with $\#\RR_\ell \lesssim \norm{u}{\A_s}^{1/s}\,\eta_\ell(u_\ell)^{-1/s}$  and $\theta\,\eta_\ell(u_\ell)^2 \le \eta_\ell(\RR_\ell,u_\ell)^2$.  According to step~(iii)--(iv) of Algorithm~\ref{algorithm}, we infer
\begin{align*}
 \#\MM_\ell \lesssim \#\RR_\ell \lesssim \norm{u}{\A_s}^{1/s}\,\eta_\ell(u_\ell)^{-1/s}
 \quad\text{for all }\ell\ge\ell_0.
\end{align*}
Linear convergence~\eqref{eq:convergence} yields that $\eta_\ell(u_\ell) \lesssim \qlin^{\ell-j}\,\eta_j(u_j)$.
Hence, it follows that
\begin{align*}
 \eta_j(u_j)^{-1/s} \lesssim \qlin^{(\ell-j)/s}\,\eta_\ell(u_\ell)^{-1/s}
 \quad\text{for all }\ell_0 \le j \le \ell.
\end{align*}
Newest vertex bisection guarantees the mesh-closure estimate
\begin{align}\label{nvb:mesh_closure}
\#\TT_\ell- \#\TT_0 + 1 
 \le \Ccls \sum_{j=0}^{\ell-1}\#\MM_j,
\end{align}
where the hidden constant depends only on $\#\TT_0$; see~\cite{bdd04,stevenson:nvb,kpp13}. Note that
\begin{align*}
\sum_{j=0}^{\ell-1}\#\MM_j
 = \sum_{j=0}^{\ell_0-1}\#\MM_j + \sum_{j=\ell_0}^{\ell-1}\#\MM_j
 &\le \big(\ell_0\,C(\ell_0) +1 \big)\sum_{j=\ell_0}^{\ell-1}\#\MM_j\\
 &\text{with }
C(\ell_0) := \max_{j=0,\dots,\ell_0-1} \frac{\#\MM_j}{\#\MM_{\ell_0}}.
\end{align*}
Together with the geometric series for $0<\qlin^{1/s}<1$, this yields that
\begin{align*}
 \#\TT_\ell- \#\TT_0 + 1 
 \lesssim \sum_{j=\ell_0}^{\ell-1}\#\MM_j
 \lesssim \norm{u}{\A_s}^{1/s} \sum_{j=\ell_0}^{\ell-1} \eta_j(u_j)^{-1/s}
 \lesssim \norm{u}{\A_s}^{1/s} \eta_\ell(u_\ell)^{-1/s}.
\end{align*}
Altogether, we thus see that
\begin{align*}
\eta_\ell(u_\ell) 
\le \Copt\,\norm{u}{\A_s} (\#\TT_\ell-\#\TT_0 + 1)^{-s}
\quad\text{for all }\ell\ge\ell_0.
\end{align*} 
Tracking the constants in the above estimates, we reveal that
\begin{align*}
 \Copt = \Clin \bigg(\frac{2\,\Cmark\,C\,\Ccls\,\big(\ell_0\,C(\ell_0)+1\big)}{1-\qlin^{1/s}}\bigg)^s,
\end{align*}
where $C>0$ is the constant of Lemma~\ref{lemma:comparison}. This concludes the proof.
\end{proof}

\section{Numerical experiments}
It is a non trivial task to find suitable numerical test examples to test
a FEM discretization with a SUPG stabilization. 
Proposals for examples with known analytical solutions
do not characterize solutions of convection-dominated problems~\cite{Augustin:2011-1}, 
and
even the standard FEM provides reasonably good solutions.
A typical solution for such problems possess layers, which can be forced
through changes in the boundary condition. However, such solutions can be 
generated only with an unknown analytical
expression of the solution. 
Nevertheless, we consider two test examples in 2D, where the analytical solution 
is known, to confirm Theorem~\ref{theorem:convergence} and 
Theorem~\ref{theorem:rates}. For the third example, we do not know 
the analytical solution.
All example are convection-dominated. Thus, we apply the SUPG FEM~\eqref{eq:supg}
with the stabilization parameter
\begin{align*}
 \vartheta_T &:= \begin{cases}
   \frac{h_T}{p\norm{\alpha}{L^{\infty}(\Omega)}}
   \quad&\text{if }\Pe_T > 1\quad\text{(convection-dominated case)},\\
  \frac{h_T^2}{2\eps p^2}\quad&\text{if }\Pe_T \le 1\quad\text{(diffusion-dominated case)}.
 \end{cases}
\end{align*}
Here, we choose $p=1$ for $\PP^1$-SUPG FEM and $p=2$ for $\PP^2$-SUPG FEM.

Note that our Algorithm~\ref{algorithm} is not a classical refinement 
algorithm known from the literature. As a refinement parameter
we use $\theta=1$ for uniform mesh-refinement (uni.) and $\theta=0.5$ for
adaptive mesh-refinement (ada.). We choose newest vertex bisection
with three bisections for refinement.
In the convergence plots the convergence order $\OO(N^{-s})$, $s>0$, 
corresponds to $\OO(h^{2s})$ for uniform mesh-refinement, where
$N$ is the number of elements of $\TT_\bullet$ and
$h:=\max_{T\in\TT_\bullet}h_T$ is the maximal mesh size.  
Therefore, the maximal possible convergence order for the
error in the energy norm $\enorm{u-u_\bullet}$ and 
the SUPG norm $\enorm{u-u_\bullet}_{\bullet,supg}$ are
$\OO(N^{-1/2})$ for $\PP^1$-SUPG FEM and $\OO(N^{-1})$ for $\PP^2$-SUPG FEM.

\begin{figure}
\begin{center}
	\subfigure[\label{subfig:bsp1P1error}$\PP^1$-SUPG FEM.]
	{\includegraphics[width=0.45\textwidth]
	{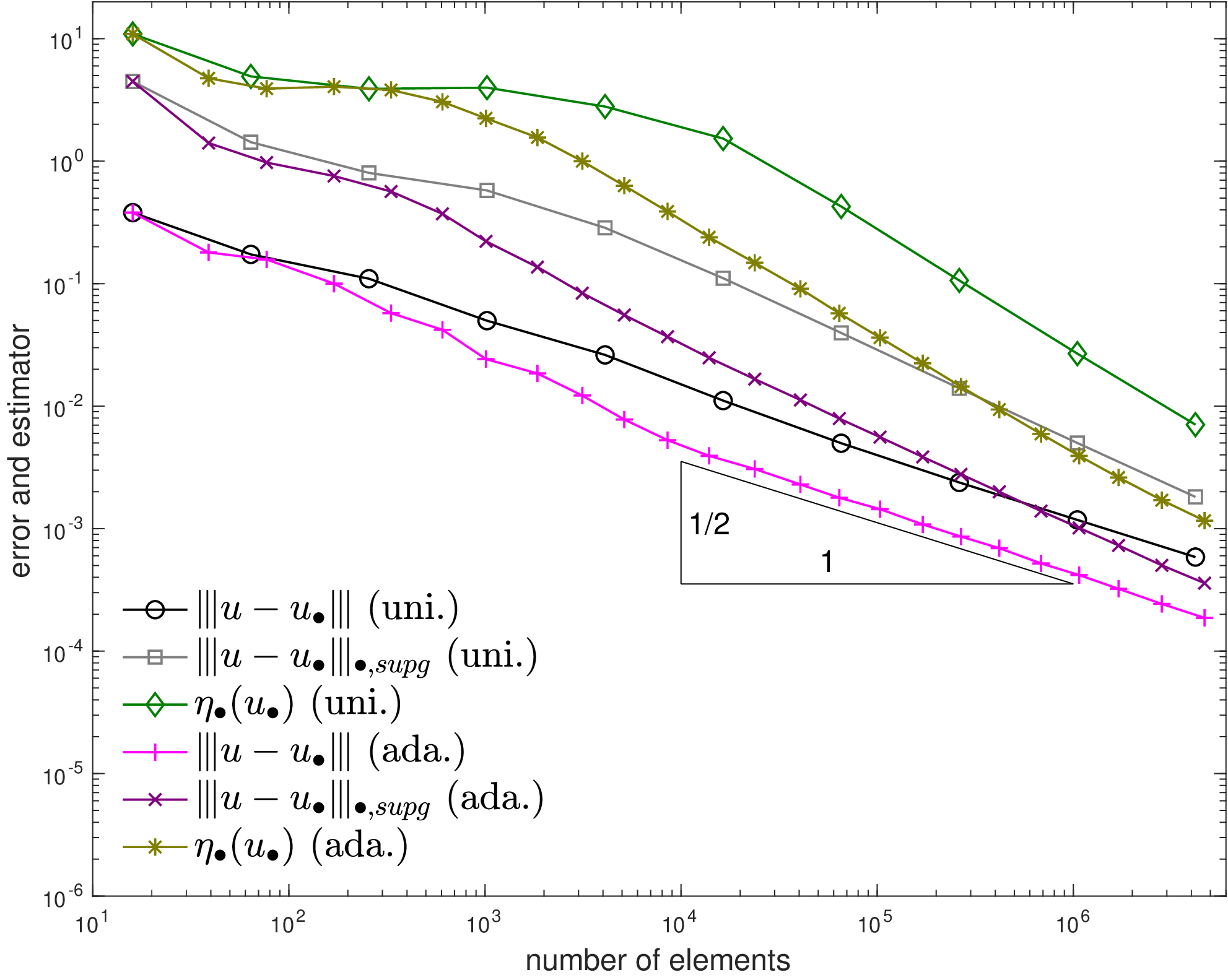}}
 \hspace{0.05\textwidth}
	\subfigure[\label{subfig:bsp1P2error}$\PP^2$-SUPG FEM.]
	{\includegraphics[width=0.45\textwidth]
	{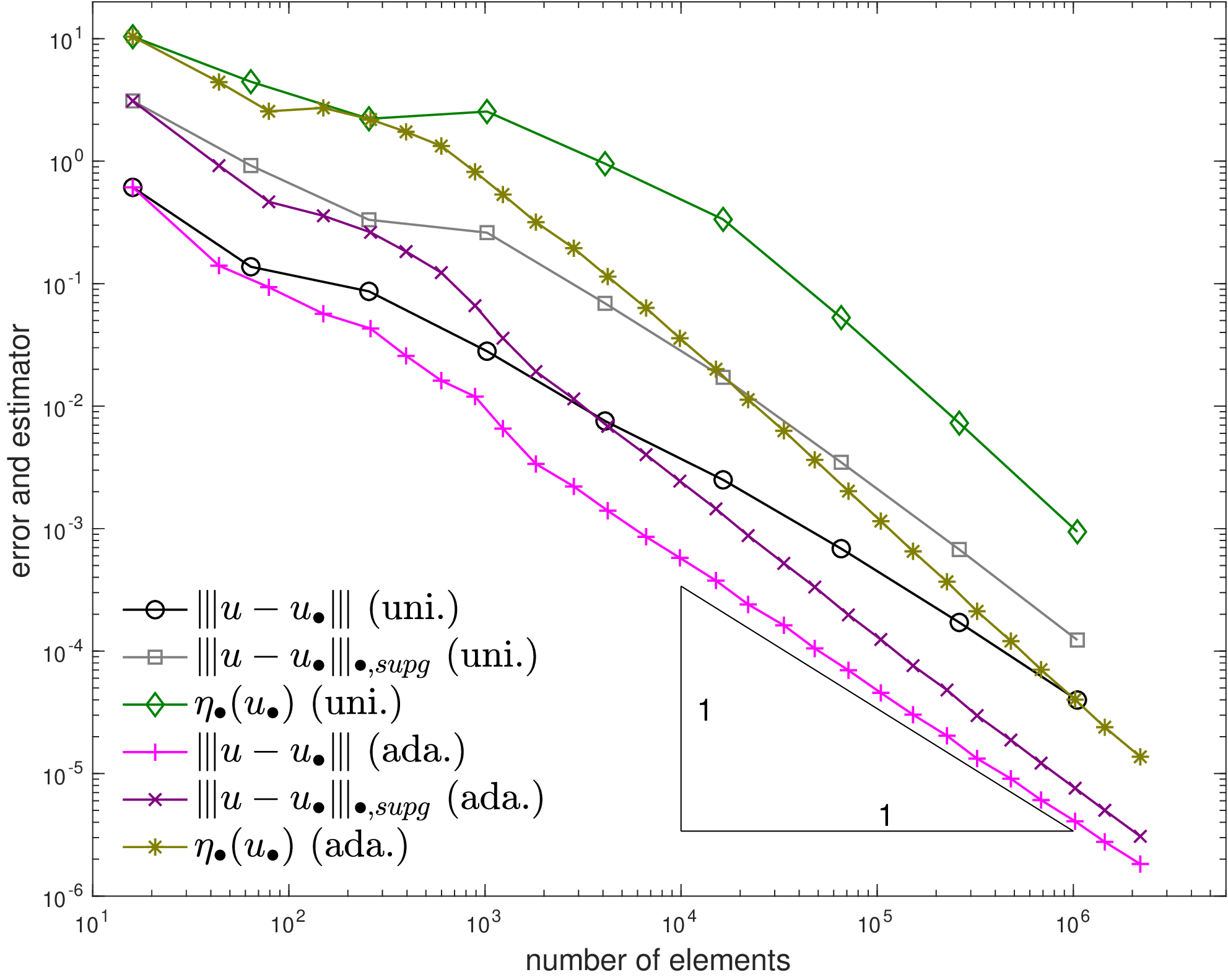}}
\end{center}
\caption{\label{fig:bsp1}%
Experiment with known smooth solution from Section~\ref{ex:bsp1}.
}
\end{figure}
\subsection{Experiment with known smooth solution}
\label{ex:bsp1}
In this example, taken from~\cite{John:2013-1}, we want to prescribe
a solution on $\Omega=(0,1)^2$ with a circular interior layer, i.e., 
\begin{align*}
 u(x_1,x_2)=16\,x_1\,(1-x_1)\,x_2\,(1-x_2)
 \bigg(\frac{1}{2}+\textstyle\frac{\text{\normalsize$\arctan$}
 \big[2\eps^{-1/2}(0.25^2-(x_1-0.5)^2-(x_2-0.5)^2)\big]}{\displaystyle\pi}\bigg).
\end{align*}
For the model data, we choose $\eps=10^{-4}$, $\a=(2,3)^T$, and $\b=2$.
The whole boundary $\partial \Omega$ is of Dirichlet type, i.e., $\partial \Omega=\Gamma_D$,
and for this example homogeneous.
We calculate the right-hand side $f$ appropriately. The uniform initial mesh $\TT_0$
consists of $16$ congruent triangles.
Since $u$ is smooth, we see
in Figure~\ref{fig:bsp1} that both, uniform and adaptive mesh-refinement,
lead to optimal convergence behavior in the asymptotic, i.e.,
$\OO(N^{-1/2})$  for $\PP^1$-SUPG FEM and $\OO(N^{-1})$ 
for $\PP^2$-SUPG FEM.
However, the absolute values of the errors for adaptive refinement
are smaller because the steep
circular interior layer is refined much faster. The influence of
the P\'eclet number on the efficiency constant
$\Ceff$ in~\eqref{eq:efficient} decreases faster.
We also refer to the discussion about semi-robustness and robustness
of the {\sl a~posteriori} estimate in Remark~\ref{rem:robustness}.

\begin{figure}
\begin{center}
	\subfigure[\label{subfig:bsp2P1error}$\PP^1$-SUPG FEM.]
	{\includegraphics[width=0.45\textwidth]
	{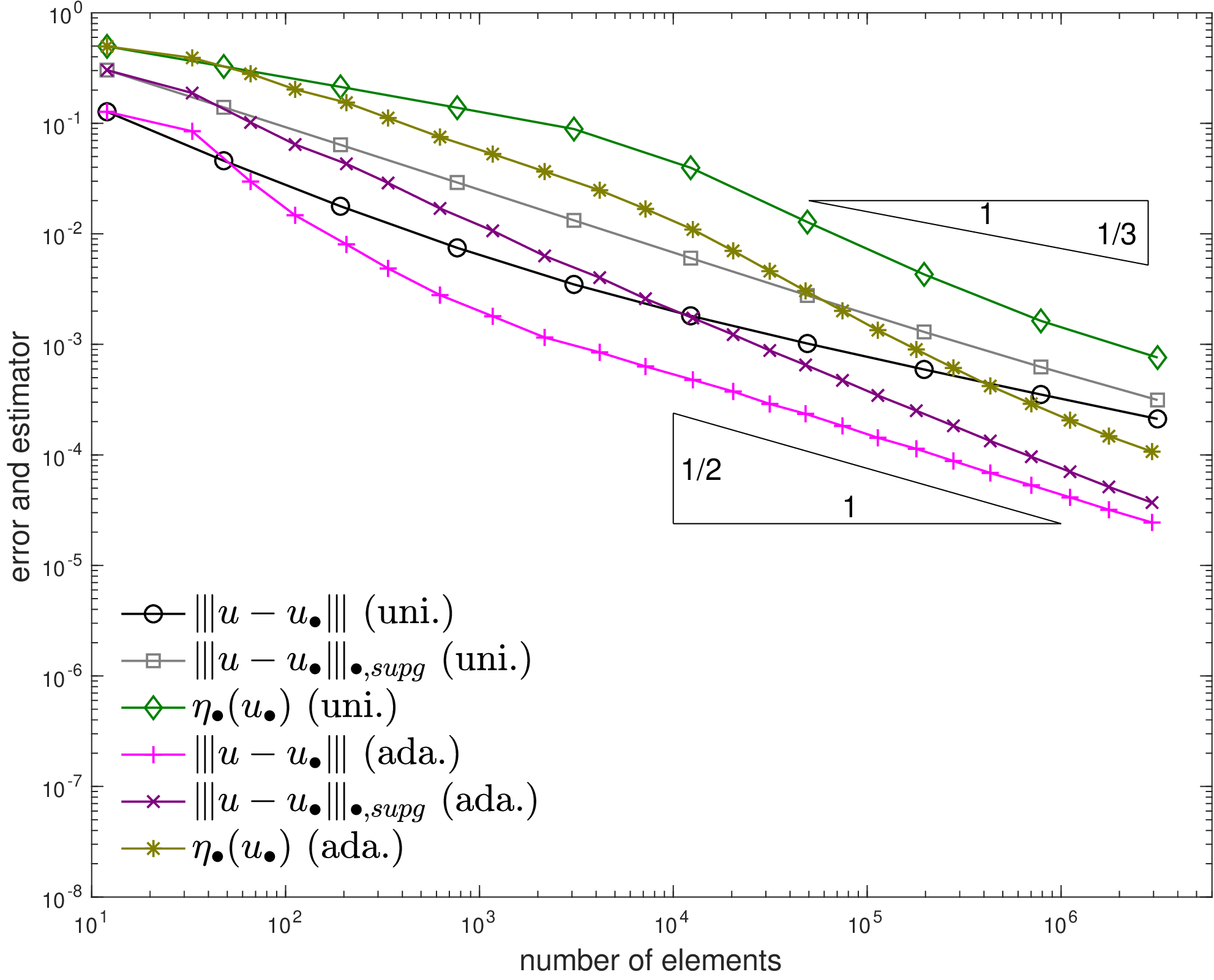}}
 \hspace{0.05\textwidth}
	\subfigure[\label{subfig:bsp2P2error}$\PP^2$-SUPG FEM.]
	{\includegraphics[width=0.45\textwidth]
	{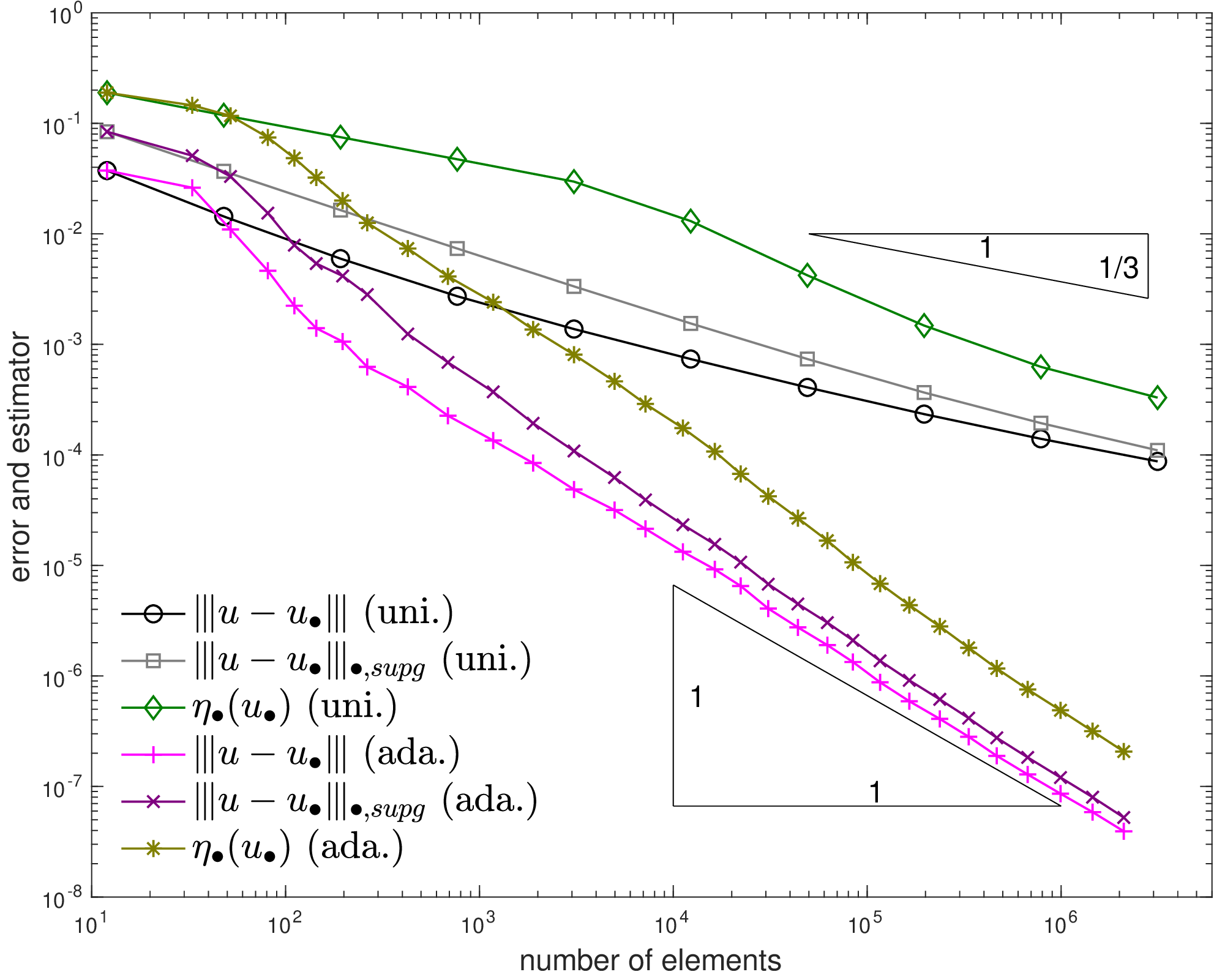}}
\end{center}
\caption{\label{fig:bsp2error}%
Experiment with known singular solution from Section~\ref{ex:bsp2}.}
\end{figure}
\subsection{Experiment with known solution with a generic singularity at the reentrant corner}
\label{ex:bsp2}
We consider a classical L-shaped domain $\Omega=(-1,1)^2\backslash([0,1]\times [-1,0])$.
The exact solution 
$u(x_1,x_2) = r^{2/3}\sin(2\varphi/3)$ 
in polar coordinates $r\in\R_0^+$, $\varphi\in[0,2\pi[$,
and $(x_1,x_2) = r(\cos\varphi,\sin\varphi)$ on $\Omega$ is used as 
a classical example to test the performance
of an adaptive algorithm.
It is well known that $u$ has a generic singularity at the reentrant corner $(0,0)$, which
leads to $u\in H^{1+2/3-\varepsilon}(\Omega)$ for all $\varepsilon>0$.
For the model data we choose $\eps=10^{-3}$, $\a=(2,3)^T$, and $\b=2$.
We allow mixed boundary condition with Neumann boundary on $[-1,1]\times 1$ and
$1\times[0,1]$. 
Note that in extension of our model problem~\eqref{eq:strongform} 
and our theory the Dirichlet boundary is inhomogeneous to guarantee that
$\Gamma_D$ contains the inflow boundary.
We calculate the right-hand side $f$ appropriately.
The uniform initial mesh $\TT_0$
consists of $12$ congruent triangles.
In Figure~\ref{fig:bsp2error}, we see the convergence rates.
For uniform mesh-refinement, we get a suboptimal convergence rate
of $\OO(N^{-1/3})$ in the asymptotic as predicted by approximation theory.
Our Algorithm~\ref{algorithm} recovers the optimal convergence rates for $\PP^1$-SUPG FEM
and $\PP^2$-SUPG FEM.

\begin{figure}
\begin{center}
	\subfigure[\label{subfig:bsp3P1error}$\PP^1$-SUPG FEM.]
	{\includegraphics[width=0.45\textwidth]
	{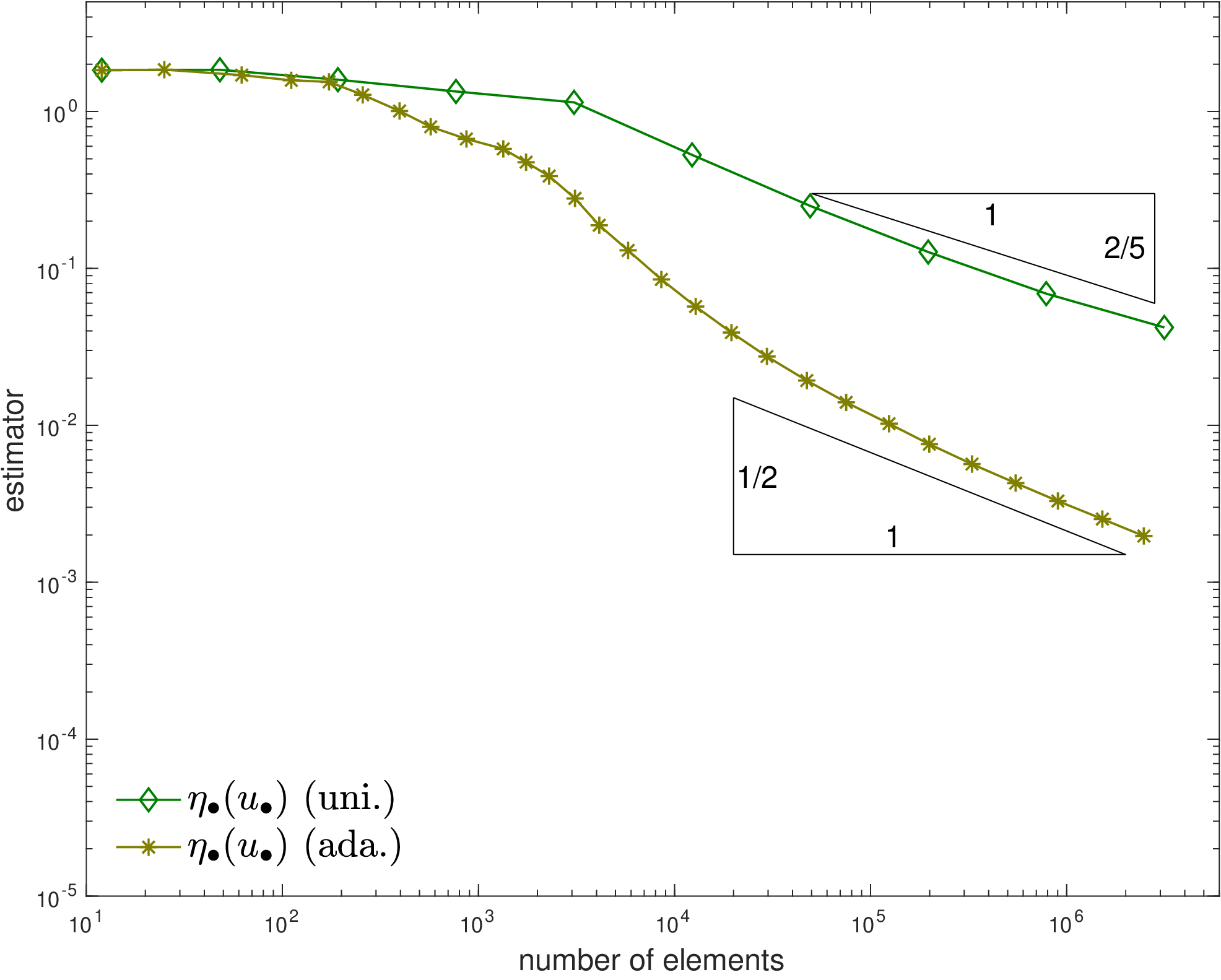}}
 \hspace{0.05\textwidth}
	\subfigure[\label{subfig:bsp3P2error}$\PP^2$-SUPG FEM.]
	{\includegraphics[width=0.45\textwidth]
	{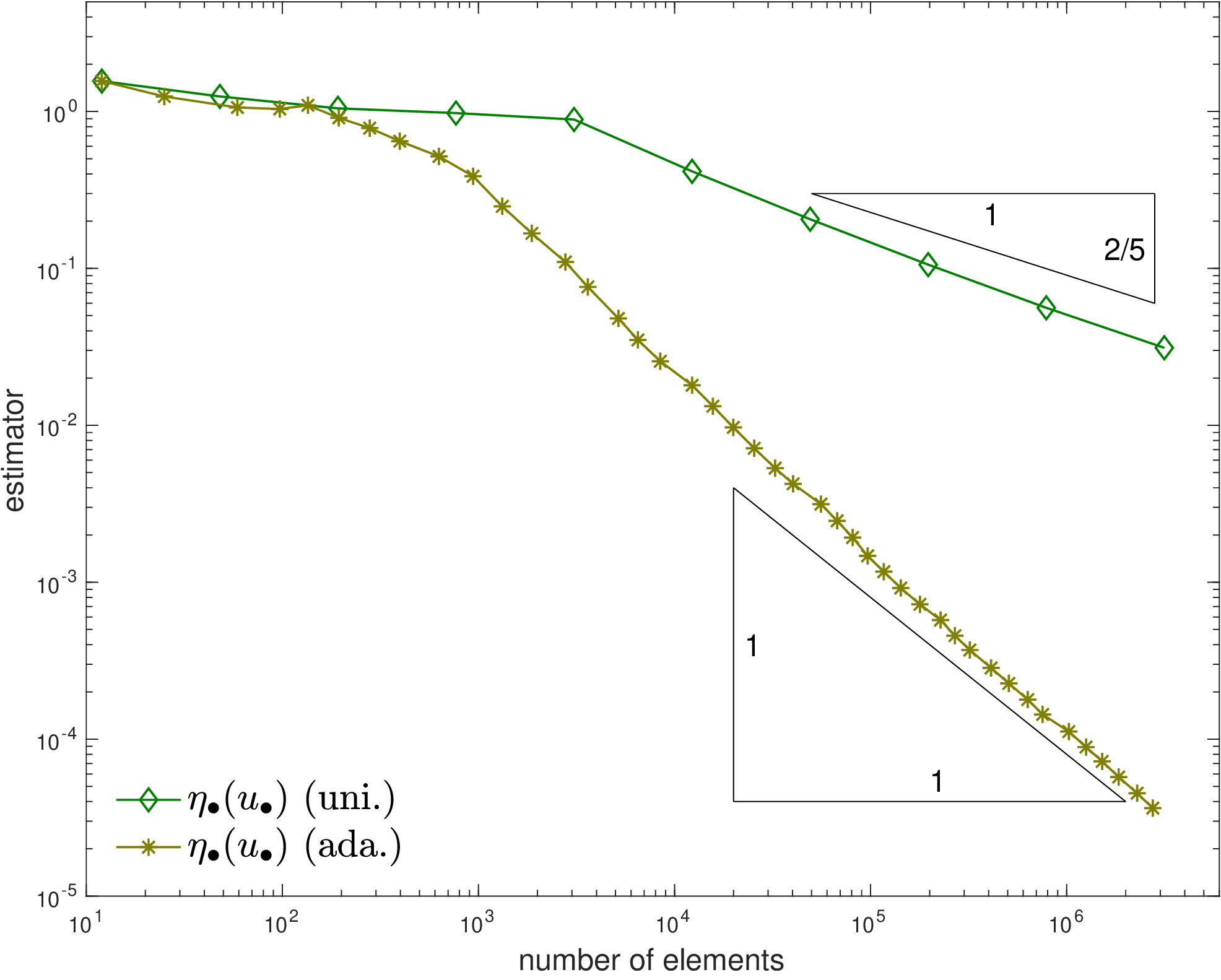}}
\end{center}
\caption{\label{fig:bsp3error}%
Experiment with unknown solution and a reentrant corner in $(0,0)$ 
from Section~\ref{ex:bsp3}.}
\end{figure}
\begin{figure}
\begin{center}
	\subfigure[\label{subfig:bsp3P1mesh}$\TT_{14}$ ($5800$ elements), $\PP^1$-SUPG FEM.]
	{\includegraphics[width=0.425\textwidth]
	{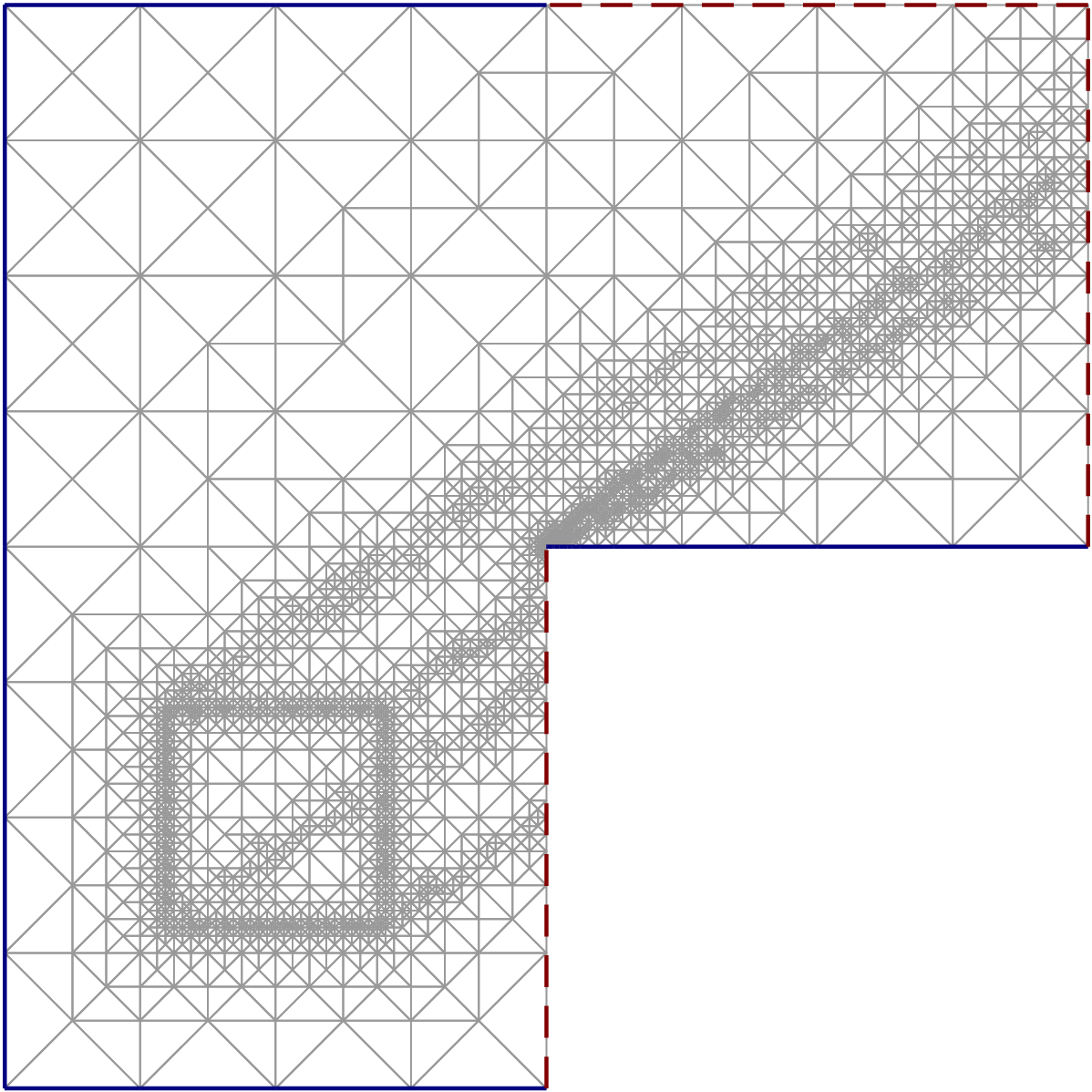}}
 \hspace{0.075\textwidth}
	\subfigure[\label{subfig:bsp3P2mesh}$\TT_{14}$ ($5176$ elements), $\PP^2$-SUPG FEM.]
	{\includegraphics[width=0.425\textwidth]
	{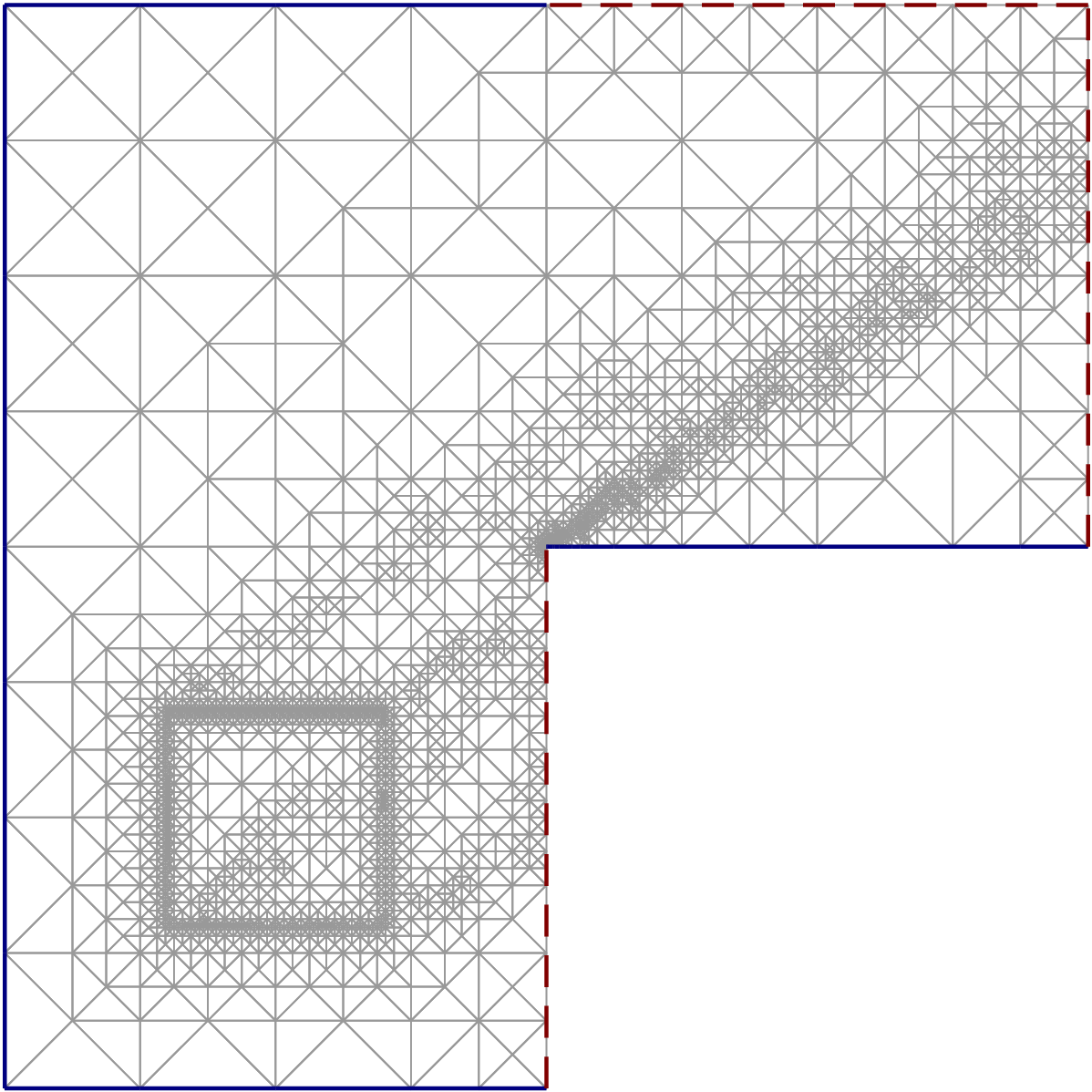}}\\[1em]
	\subfigure[\label{subfig:bsp3P1sol}Solution on $\TT_{14}$, $\PP^1$-SUPG FEM.]
	{\includegraphics[width=0.45\textwidth]
	{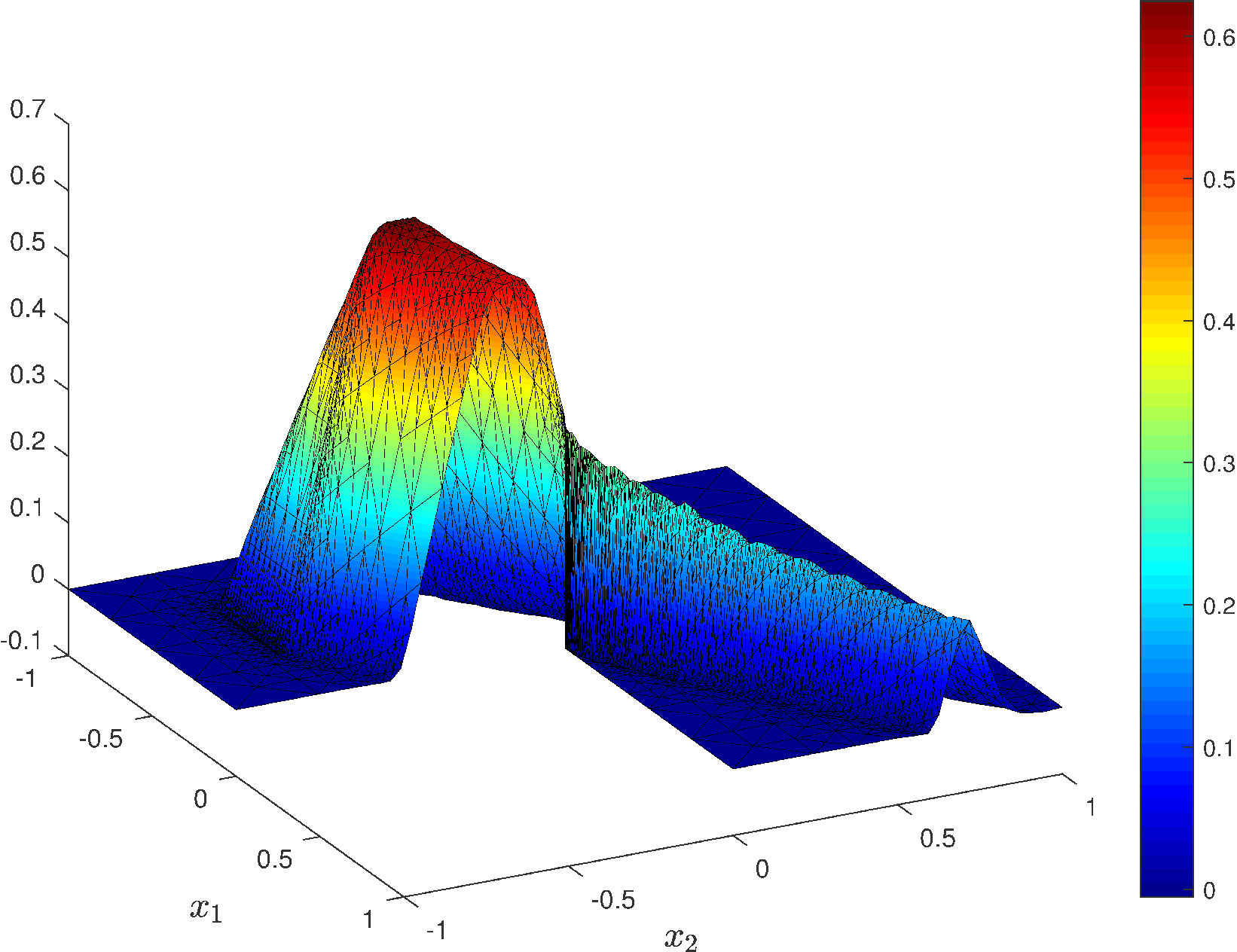}}
 \hspace{0.05\textwidth}
	\subfigure[\label{subfig:bsp3P2sol}Solution on $\TT_{14}$, $\PP^2$-SUPG FEM.]
	{\includegraphics[width=0.45\textwidth]
	{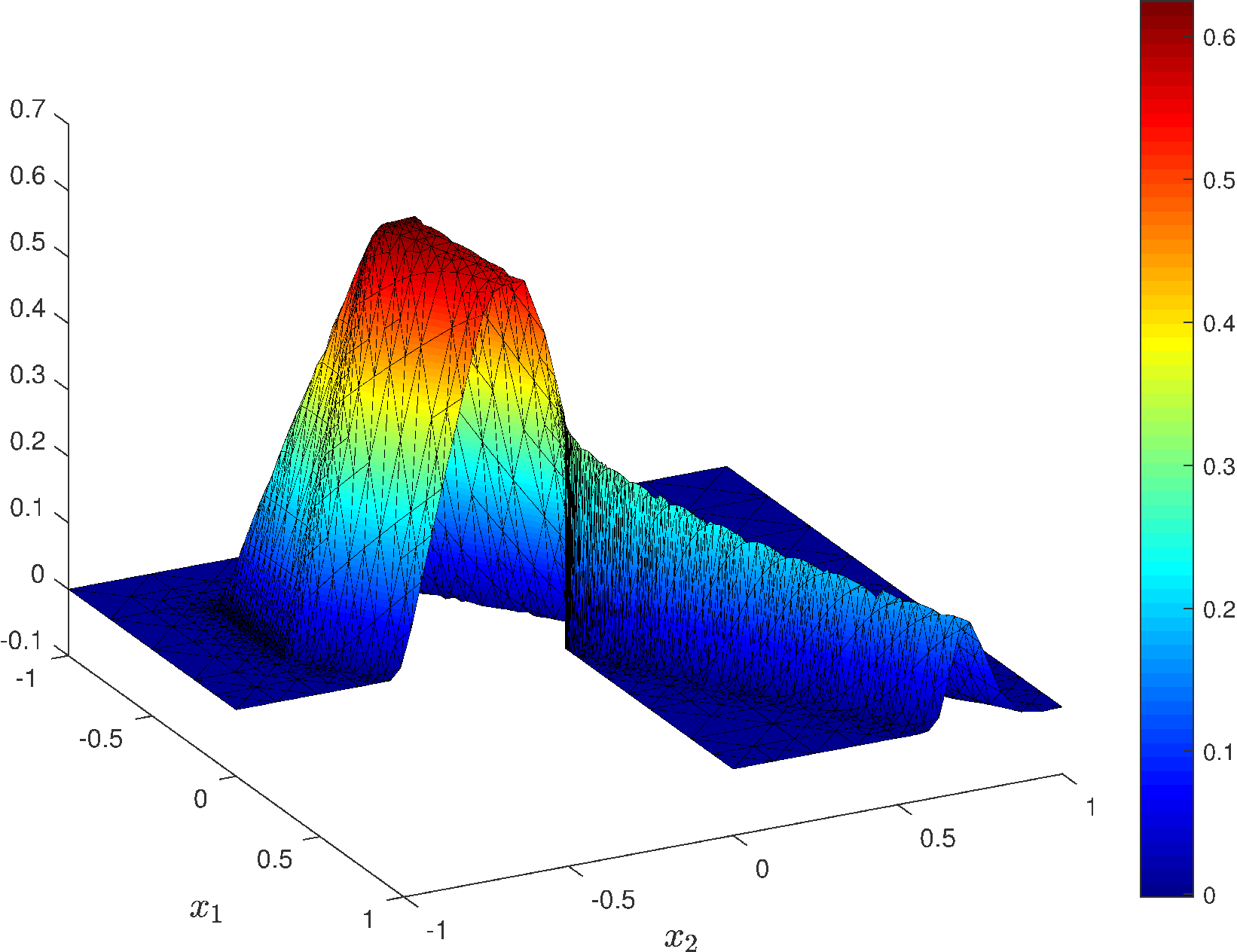}}
\end{center}
\caption{\label{fig:bsp3solmesh}%
Adaptive mesh (above) and solution for the experiment 
from Section~\ref{ex:bsp3} after $14$ refinements. 
The dashed lines represent the outflow Neumann boundary.  
}
\end{figure}

\subsection{A more practical example}
\label{ex:bsp3}
We consider the L-shaped domain from section~\ref{ex:bsp2}.
Furthermore, the model data are $\eps=10^{-3}$, $\a=(3,2)^T$, and $\b=1$. 
For the right-hand side,
we set $f=5$ in the square $[-0.7,-0.3]^2$, and $f=0$ otherwise.
The mixed boundary condition is prescribed as follows; see
also Figure~\ref{subfig:bsp3P1mesh} and Figure~\ref{subfig:bsp3P2mesh}:
we set the Neumann data $g=10^{-3}$ on the lines $0\times[-1,0]$, $1\times[0,1]$, and $[0,1]\times 1$.
The rest of the boundary is a homogenous Dirichlet boundary.
Note that this model setting fulfills exactly the requirements of the
model problem~\eqref{eq:strongform}.
Since the analytical solution is unknown we
plot only the estimator in Figure~\ref{fig:bsp3error}. Furthermore, we remind that the
estimator is reliable and efficient. Thus, we observe in fact 
the convergence rate of our numerical solution by the asymptotics
of the estimator. 
Applying the adaptive Algorithm~\ref{algorithm}, we also know that
our numerical solution converges with the best possible rate due 
to Theorem~\ref{theorem:convergence} and Theorem~\ref{theorem:rates}.
As in the example
in section~\ref{ex:bsp2}, uniform mesh-refinement leads to suboptimal convergence rates,
whereas our adaptive Algorithm~\ref{algorithm} reproduces the optimal convergence
rates for $\PP^1$-SUPG FEM as well as for $\PP^2$-SUPG FEM.  
In Figure~\ref{fig:bsp3solmesh}, we see an adaptively generated mesh $\TT_{14}$
and the solution. The mesh-refinement is mainly around the source,
from the source in convection direction, and at the reentrant corner $(0,0)$.

\section{Conclusions}
The numerical method as well as the corresponding {\sl a~posteriori} estimators
to approximate
convection-dominated convection-diffusion problems 
must be chosen very carefully.
In this work, we showed asymptotically the optimal convergence of an adaptive
finite element algorithm with SUPG stabilization.
Since the numerical scheme does not have a classical Galerkin orthogonality, some new
ideas were implemented. 

\bigskip

\textbf{Acknowledgement.}\quad
The second author is supported by the Austrian Science Fund (FWF) 
through the research project \textit{Optimal adaptivity for BEM and FEM-BEM coupling} 
(grant P27005) and the special research 
program \textit{Taming complexity in PDE systems} (grant SFB F65).

\bibliographystyle{alpha}
\bibliography{literature}

\end{document}